\documentclass[11pt]{amsart}

\usepackage{amsfonts,amsthm,amsmath,enumerate,amssymb,mathrsfs,latexsym,color,bm,bbm}
\usepackage[text={6in,9in},centering]{geometry}
\usepackage[backref=page,
            colorlinks,
            linkcolor=blue,
            anchorcolor=red,
            citecolor=red]{hyperref}
\usepackage{tikz}
\usepackage{verbatim, soul}
\usetikzlibrary{arrows.meta,calc,positioning}

\theoremstyle{plain}
\newtheorem{theorem}{Theorem}[section]
\newtheorem{lemma}[theorem]{Lemma}
\newtheorem{corollary}[theorem]{Corollary}
\newtheorem{proposition}[theorem]{Proposition}
\newtheorem{conjecture}[theorem]{Conjecture}

\theoremstyle{definition}
\newtheorem{definition}[theorem]{Definition}

\theoremstyle{remark}
\newtheorem{remark}[theorem]{Remark}

\numberwithin{equation}{section}

\DeclareMathOperator{\dimh}{dim_H}

\DeclareMathOperator{\spt}{spt}
\DeclareMathOperator{\dist}{dist}
\DeclareMathOperator{\diam}{diam}
\DeclareMathOperator{\diag}{diag}

\newcommand{\rd}{\,\mathrm{d}}

\def\R{\mathbb R}
\def\Z{\mathbb Z}
\def\T{\mathbb T}
\def\N{\mathbb N}
\def\Q{\mathbb Q}
\def\D{\mathcal D}

\def\h{\mathcal H}
\def\i{\bm i}

\def\u{\bm u}

\def\p{\bm p}
\def\a{\bm a}
\def\b{\bm b}
\def\hl{\widehat \ell}

\makeatletter
\def\l@subsection{\@tocline{2}{0pt}{2.5pc}{5pc}{}}
\def\l@subsubsection{\@tocline{2}{0pt}{5pc}{7.5pc}{}}
\makeatother

\begin{document}

\title[Positive-measure intersections of Bedford--McMullen carpets]{Positive-measure intersections of Bedford--McMullen carpets under similitudes}

\author{Ding-Kun Hu}
\address{School of Mathematical Sciences, Zhejiang University, Hangzhou, China}
\email{12635001@zju.edu.cn}

\author{Huo-Jun Ruan}
\address{School of Mathematical Sciences, Zhejiang University, Hangzhou, China}
\email{ruanhj@zju.edu.cn}

\author{Jian-Ci Xiao$^{\ast}$}
\address{School of Mathematics, Nanjing University of Aeronautics and Astronautics, Nanjing, China}
\email{jcxiao@nuaa.edu.cn}
\thanks{$^\ast$Corresponding author.}

\subjclass[2020]{Primary 28A80; Secondary 28A75}
\keywords{Bedford--McMullen carpet, self-affine measure, positive-measure intersection, oblique similitude, tangent set}

\begin{abstract}
    Let $K=K(n,m,\Gamma)$ be a Bedford--McMullen carpet with $n>m\geq2$ and $1<\#\Gamma<nm$, and let $\mu$ be a self-affine measure on $K$ with strictly positive weights.  We prove that every similitude $f(z)=rOz+a$ satisfying $\mu(K\cap f(K))>0$ has a diagonal or anti-diagonal orthogonal part $O$, without any arithmetic restrictions on the defining bases. Under the same intersection hypothesis, we also show that $r\in n^{\Q}\cup m^{\Q}$ whenever $K$ is not a segment. For isometries satisfying this hypothesis, we further prove that $f(K)$ is simply a translate of $K$, thereby establishing a strengthened form of the 2010 conjecture of Elekes, Keleti and M\'ath\'e in this planar self-affine context. 
\end{abstract}

\maketitle

\section{Introduction}\label{sec:introduction}
\subsection{Background and main results}

Let $K\subset\R^d$ be a self-similar or self-affine set and let $\mu$ be a Borel probability measure on $K$. A basic rigidity question asks which similitudes $f$ satisfy the condition $\mu(K\cap f(K))>0$. In particular, one asks whether the positive-measure intersection forces $f$ to respect the geometric structure of $K$, and how such rigidity depends on the algebraic properties of the maps generating the set.

For homogeneous self-similar sets, a connection between intersections and the defining contractions dates back to Furstenberg's intersections conjecture, which was independently verified by Shmerkin~\cite{Shm19} and Wu~\cite{Wu19}. Feng, Huang, and Rao~\cite{FHR14} investigated a weaker form in the inhomogeneous case. Let $E, F\subset\R^d$ be nontrivial self-similar sets. Suppose the iterated function system generating $E$ satisfies the open set condition and $\dimh F$ equals its similarity dimension. They proved that, if $F$ cannot be affinely embedded into $E$, then
\[
    \dimh(E\cap g(F))<\dimh F
\]
for every $C^1$-diffeomorphism $g$ of $\R^d$. Thus an intersection with Hausdorff dimension equal to $\dimh F$ forces an affine embedding. They further conjectured that affine embeddings between totally disconnected self-similar sets impose a logarithmic commensurability relation on the associated contraction ratios. For recent progress on this conjecture, see~\cite{Alg20-2, AHW26, FX18} and references therein.

Complementing these dimensional results, Elekes, Keleti and M\'ath\'e~\cite{EKM10} studied positive-measure intersections of self-similar and self-affine sets. For a self-similar set satisfying the strong separation condition, they proved that its intersection with a similar copy has positive self-similar measure if and only if the intersection has non-empty relative interior~\cite[Theorem~4.5]{EKM10}. They also established the corresponding statement for translates of self-affine Sierpi\'nski sponges equipped with the uniform self-affine measure, and further conjectured that for such sponges, every positive-measure intersection with an isometric copy is equal to an intersection with a translate.

We resolve this conjecture for Bedford--McMullen carpets~\cite{Bed84,Mcm84}, whose construction distinguishes two coordinate directions with different contraction ratios. Let us first recall the definition. For an integer $q\geq1$, write $[q]:=\{0,1,\ldots,q-1\}$.

\begin{definition}\label{def:bmcarpet}
    Let $n>m\geq 2$ be integers and let $\Gamma\subset[n]\times[m]$ satisfy $1<\#\Gamma<mn$ (where $\#$ denotes cardinality). For $(i,j)\in [n]\times[m]$, let 
    \[
        \phi_{i,j}(x,y) = \Big( \frac{x+i}{n},\frac{y+j}{m} \Big), \quad (x,y)\in\R^2.
    \]
    The \emph{Bedford--McMullen carpet}  associated with $\{\phi_{i,j}\}_{(i,j)\in\Gamma}$ is the unique non-empty compact set $K=K(n,m,\Gamma)\subset\R^2$ satisfying $K=\bigcup_{(i,j)\in\Gamma} \phi_{i,j}(K)$.
\end{definition}

For any such carpet $K=K(n,m,\Gamma)$ and any probability vector $\p=(p_{i,j})_{(i,j)\in\Gamma}$ with strictly positive entries, let $\mu_{\p}$ be the self-affine measure on $K$ determined by
\[
    \mu_{\p}=\sum_{(i,j)\in\Gamma}p_{i,j}\phi_{i,j}\mu_{\p}.
\]
Here $g\lambda$ denotes the pushforward of a measure $\lambda$ under a Borel map $g$, that is, $g\lambda(E)=\lambda(g^{-1}(E))$. Since the weights are strictly positive, $\spt\mu_{\p}=K$. We call $\mu_{\p}$ \emph{uniform} when $p_{i,j}=1/\#\Gamma$ for every $(i,j)\in\Gamma$.

The conjecture~\cite[Conjecture 9.1]{EKM10} was formulated for self-affine Sierpi\'nski sponges in arbitrary dimensions. In the planar self-affine setting, it takes the following form.

\begin{conjecture}[\cite{EKM10}]
    Let $K=K(n,m,\Gamma)$ be a Bedford--McMullen carpet and let $\mu$ be its uniform self-affine measure. For every isometry $f$, if $\mu(K\cap f(K))>0$ then $K\cap f(K)=K\cap (K+t)$ for some $t\in\R^2$.
\end{conjecture}

Our first main result establishes this conjecture in a strengthened form. It identifies the entire isometric copy as a translate and allows arbitrary self-affine measures.

\begin{theorem}\label{thm:ekmconj}
    Let $K=K(n,m,\Gamma)$ be a Bedford--McMullen carpet and let $\mu_{\p}$ be a self-affine measure on $K$ with positive weights. For any isometry $f$, if $\mu_{\p}(K\cap f(K))>0$, then $f(K)=K+t$ for some $t\in\R^2$. 
\end{theorem}

The key ingredient of our proof is a restriction on the orthogonal part of an arbitrary similitude. Note that a similitude $f$ on $\R^2$ has the form $f(z)=rOz+a$, where $r>0$ is the \emph{similarity ratio}, $O\in O(2)$ is the \emph{orthogonal part}, and $a\in\R^2$ is the \emph{translation part}. Following~\cite{Xiao26}, we call $f$ \emph{oblique} if $O$ is neither diagonal nor anti-diagonal. Thus a non-oblique similitude preserves the horizontal and vertical directions as an unordered pair.

Previous work established this non-oblique restriction for self-embeddings. Under the multiplicative independence condition $\frac{\log n}{\log m}\notin\Q$, Algom and Hochman~\cite{AH19} proved that every similitude sending $K$ into itself is an isometry with diagonal orthogonal part, provided either that the carpet is not a product or that both factors have Hausdorff dimensions strictly between zero and one. The third author~\cite{Xiao26} subsequently proved non-obliqueness for all self-embeddings, assuming only that the carpet is not contained in a line. 

Our second main result shows that the positive-measure intersection already forces non-obliqueness.

\begin{theorem}\label{thm:nonoblique}
	Let $K=K(n,m,\Gamma)$ be a Bedford--McMullen carpet and let $\mu_{\p}$ be a self-affine measure on $K$ with positive weights. Assume that a similitude $f(z)=rOz+a$ on $\R^2$ satisfies $\mu_{\p}(K\cap f(K))>0$. Then $O$ is either diagonal or anti-diagonal.
\end{theorem}

Neither theorem requires the multiplicative independence condition between $n$ and $m$, nor does either assume the strong separation condition. Theorem~\ref{thm:nonoblique} recovers the non-obliqueness statement for self-embeddings by applying it to $f^{-1}$. Indeed, if $f(K)\subset K$, then $K\subset f^{-1}(K)$, and hence $\mu_{\p}(K\cap f^{-1}(K))=1$. The hypothesis $n>m$ is essential for Theorem~\ref{thm:nonoblique}. In the self-similar case $n=m$, a generalized Sierpi\'nski carpet may admit an oblique reflection preserving it; see~\cite[Example 6.1]{Xiao26}.

\begin{remark}
    Our proof provides further information when $f$ is an axis-interchanging isometry. In this case, a positive-measure intersection forces $\frac{\log n}{\log m}\in\Q$ and $K$ to be a product set; see Subsection~\ref{subsec:axis-interc-isom}. Such a phenomenon occurs. For example, the product of the middle-third Cantor set with itself has a Bedford--McMullen representation with $n=9$, $m=3$ and is invariant under coordinate interchange.
\end{remark}

We also obtain an arithmetic restriction on the similarity ratio. The following theorem extends the logarithmic commensurability results for self-embeddings in~\cite{AH19,Xiao26} to the present setting. For carpets that are not product sets with multiplicatively independent bases, it extends the isometry conclusion of~\cite[Theorem 1.1]{AH19}.


\begin{theorem}\label{thm:log}
    Under the assumptions of Theorem~\ref{thm:nonoblique}, if in addition $K$ is not a segment, we have the following arithmetic restriction on the similarity ratio.  
    \begin{enumerate}
        \item If $O$ is anti-diagonal, then $\frac{\log n}{\log m}\in\Q$ and $r\in n^{\Q}=m^{\Q}$;
        \item If $O$ is diagonal, then $r\in n^{\Q}\cup m^{\Q}$. Moreover, $r=1$ provided either $K$ is not a product set, or $\frac{\log n}{\log m}\notin\Q$ and both factors have Hausdorff dimensions strictly between zero and one.
    \end{enumerate}
\end{theorem}

The arithmetic conclusion in Theorem~\ref{thm:log}(2) cannot be strengthened to $r\in n^{\Q}\cap m^{\Q}$ in general. For example, if $n=4$, $m=3$, and $\Gamma=[4]\times\{0,2\}$, then $K$ is the product of $[0,1]$ with the middle-third Cantor set. The similitude $f(x,y)=(x/3,y/3)$ maps $K$ into itself and $f(K)\supset \phi_{0,0}(K)$. So $K\cap f(K)$ has positive self-affine measure, but $1/3\notin 4^{\Q}$.

Intersection rigidity has been studied from other perspectives under different hypotheses as well. For totally disconnected self-similar and self-affine sets satisfying the weak separation condition, Deng and Wang~\cite{DW18} obtained a uniform drop in the mass of the intersection under small nontrivial perturbations, which extends~\cite[Theorem 4.1]{EKM10}. For pairs of carpets with incommensurable defining bases, Algom~\cite{Alg20} obtained dimension bounds on the intersection for affine maps preserving coordinate directions, established mutual singularity of the associated measures under dimension assumptions, and identified obstructions to affine embeddings under additional hypotheses. In this paper, we consider copies of a single carpet and obtain restrictions on the similitudes that produce positive-measure intersections.

\subsection{Proof strategy}

A self-embedding provides a global inclusion $f(K)\subset K$ and hence the iterated inclusions $f^k(K)\subset K$. This allows a direct comparison between the local geometry of $K$ and its embedded images; see~\cite{AH19,Xiao26}. Under the weaker assumption $\mu_{\mathbf{p}}(K\cap f(K))>0$, we need instead recover geometric information from the density on the intersection. The self-affine setting presents another difficulty. Normalizing a cylinder of a Bedford--McMullen carpet involves an anisotropic change of coordinates, which obstructs a direct extension of the self-similar argument.

We work with approximate squares and their rescaled limits, following the tangent set approach of Algom and Hochman~\cite{AH19}. Up to bounded diagonal changes of coordinates, these limits are products of horizontal fibres and the vertical projection of $K$. Moreover, along convergent subsequences, the limiting conditional measures have full support on the limiting sets. We choose squares in which the relative measure of the intersection tends to one, and cover their inverse images in $K$ by a uniformly bounded number of comparable approximate squares. Together with the full support property, this gives an inclusion of one limiting product set in a finite union of transformed limiting product sets;  see Proposition~\ref{prop:prescribed-cover}.

For Theorem~\ref{thm:nonoblique}, such an inclusion is incompatible with obliqueness. If every row is non-empty and no row is full, then the vertical projection of $K$ is the unit interval and every horizontal fibre has zero length. Hence the target product set contains nondegenerate vertical intervals. On the other hand, every oblique image of a source product set intersects each vertical line in a set of zero length, which contradicts the inclusion. If there exists a vacant row, we prescribe a limiting horizontal fibre associated with a chosen row and apply the one-sided boundary obstruction from~\cite{Xiao26}.

The remaining case, in which every row is non-empty and some row is full, requires a more delicate argument. For self-embeddings, a full row provides a horizontal segment, and the absence of oblique segments (see~\cite[Proposition~2.2]{Xiao26}) immediately excludes obliqueness. A positive-measure intersection may avoid that segment. Furthermore, choosing a density point gives no control over the preimage digit frequencies. We therefore estimate the distribution of the heights of the inverse images and exclude addresses whose digits spend asymptotically all their time in full rows. This allows us to select scales at which all the approximate squares in the local cover have limiting horizontal fibres of zero length. The selection is stable under small height perturbations; see Proposition~\ref{prop:good-scales}. Hence the vertical-interval obstruction applies again.

For Theorem~\ref{thm:ekmconj}, Theorem~\ref{thm:nonoblique} reduces the problem to isometries that preserve or interchange the coordinate directions. In the former case, we extend the argument of Algom and Hochman~\cite[Proposition~3.1]{AH19} from one digit set to two digit sets of equal cardinality. We then find a translation vector $b\in\R^2$ such that $OK+b$ has the same coordinate projections as $K$, and a two-pattern analogue of~\cite[Proposition~7.1]{EKM10} shows that $OK+b=K$.

Assume that the isometry $f$ interchanges the coordinate directions. When $\frac{\log n}{\log m}\notin\Q$, a positive-measure intersection yields an affine embedding between the coordinate projections, and Shmerkin's theorem~\cite[Corollary~7.6]{Shm19} rules out such an embedding. For dependent bases, we apply $f^{-1}$ to small cylinders and iterate $(x,y)\mapsto (nx,my)$ until the vertical scale is one. The horizontal scale then tends to infinity. At any fixed scale, the first coordinate depends on a distant block of the original vertical address, whereas the second depends on an initial block of the horizontal address. These blocks correspond to disjoint levels and are independent under the Bernoulli measure. Passing to a subsequence along which the translation terms converge, this independence and cylinder density yield a product set contained in the toral image of $K$, with factors isometric to the toral images of its coordinate projections. This forces $K$ itself to be a product set. Expressing both projections in a common base then completes the proof.


Finally, for Theorem~\ref{thm:log}, if some coordinate projection of $K$ is a proper deleted-digit set, the intersection hypothesis yields affine embeddings between the projections, with similarity ratios determined by $r$ and the defining bases. Feng and Wang's logarithmic commensurability theorem~\cite{FW09} and Shmerkin's embedding theorem~\cite{Shm19} then give the desired arithmetic restriction. When both coordinate projections are intervals, interchange of the coordinates is excluded by our earlier rigidity argument. Normalizing at a cylinder density point and passing to a limit, we obtain a diagonal affine self-embedding of $K$. We show that its logarithmic contraction rates relative to $n$ and $m$ must coincide, which forces $f$ to be an isometry.

\begin{remark}
    Earlier descriptions of tangent sets for certain self-affine carpets were obtained by Bandt and K\"aenm\"aki~\cite{BK13}.
    The local geometry of self-affine measures on Bedford--McMullen carpets has also been studied through their scaling scenery. Ferguson, Fraser and Sahlsten~\cite{FFS15} describe the limiting scenery of Bernoulli measures in terms of products of a coordinate marginal and a conditional fibre measure, up to diagonal rescaling.
\end{remark}

\subsection{Organization of the paper}

In Section~\ref{sec:symbolic-pre}, we fix notation and develop the density and compactness arguments for approximate squares, including the construction with a prescribed limiting fibre. Sections~\ref{sec:nonobliqueness} and~\ref{sec:EKMconj} prove Theorems~\ref{thm:nonoblique} and~\ref{thm:ekmconj}, respectively. In Section~\ref{sec:log} we prove Theorem~\ref{thm:log}.

\medskip
\paragraph{{\bf Notation}}
For $z\in \R^2$, $E\subset\R^2$ and $\delta>0$, let $B(z,\delta):=\{z'\in \R^2: |z'-z|<\delta\}$ denote the open ball of radius $\delta$ centered at $z$ and let $\mathcal{N}_\delta(E)$ denote the closed $\delta$-neighborhood of $E$. For a measure $\lambda$, let $\spt\lambda$ denote its support. We write $\dimh$ for Hausdorff dimension and write $\pi_1, \pi_2$ for the orthogonal projections onto the $x$- and $y$-axes, respectively. Finally, we write $gh$ for the composition $g\circ h$.

\section{Symbolic density and approximate squares}
\label{sec:symbolic-pre}

In this section, we derive a finite covering relation between limiting product sets from the assumption that $K\cap f(K)$ has positive measure. We begin by fixing notation, and then describe the geometry of approximate squares and show how symbolic density passes to their limits. The final proposition allows us to prescribe the horizontal fibre on the target side.

\subsection{Preliminaries}

Throughout this paper, we fix integers $n>m\geq 2$ and a Bedford--McMullen carpet $K$ as defined in Definition~\ref{def:bmcarpet}. We adopt the following row notation:
\begin{itemize}
    \item $\Gamma_j :=\{i\in[n]:(i,j)\in\Gamma\}$ is the set of horizontal digits in the $j$-th row;
    \item $J :=\{j\in[m]:\Gamma_j\neq\varnothing\}$ is the set of non-empty rows;
    \item $J_* :=\{j\in[m]: \Gamma_j=[n]\}$ is the set of full rows;
    \item $q_j :=\sum_{i\in\Gamma_j}p_{i,j}$ is the total weight of the $j$-th row. Note that $q_j>0$ for $j\in J$ and $\sum_{j\in J}q_j=1$.
\end{itemize}

Two elementary observations will be used when passing to limits of rescaled approximate squares.

\begin{lemma}\label{lem:hd-continuity}
    Let $E_k,E\subset\R^2$ be non-empty compact sets such that $E_k\to E$ in the Hausdorff metric. Let $A_k,A\in M_2(\R)$ satisfy $A_k\to A$ in operator norm, and let $t_k,t\in\R^2$ satisfy $t_k\to t$. Then $A_k(E_k+t_k)\to A(E+t)$ in the Hausdorff metric.
\end{lemma}
\begin{proof}
    Let $M:=\sup_k\|A_k\|<\infty$ and pick a large $R>0$ such that $E+t\subset\overline{B(0,R)}$. Since $E_k\to E$, there exists $\delta_k\to 0$ such that $E_k\subset\mathcal{N}_{\delta_k}(E)$ and $E\subset\mathcal{N}_{\delta_k}(E_k)$. Letting
    \[
        \varepsilon_k := M(\delta_k+|t_k-t|)+R\|A_k-A\|,
    \]
    we have $\varepsilon_k\to 0$.

    For each $k$, if $x\in E_k$ then there exists $y\in E$ such that $|x-y|\leq\delta_k$. Then 
    \begin{align*}
        |A_k(x+t_k)-A(y+t)| &\leq |A_k(x+t_k)-A_k(y+t)|+|(A_k-A)(y+t)| \\
        &\leq \|A_k\|(|x-y|+|t_k-t|) + \|A_k-A\||y+t| \\
        &\leq \varepsilon_k.
    \end{align*}
    So $A_k(E_k+t_k)\subset\mathcal{N}_{\varepsilon_k}(A(E+t))$. Similarly, one can show that $A(E+t)\subset\mathcal{N}_{\varepsilon_k}(A_k(E_k+t_k))$. Since $\varepsilon_k\to 0$, the desired Hausdorff convergence follows.
\end{proof}

\begin{lemma}\label{lem:finite-cover-limit}
    Let $p\geq 1$ and let $\lambda_k,\lambda$ be Borel probability measures on $\R^2$ such that $\lambda_k\to\lambda$ weakly. For $1\leq t\leq p$, let $F_{k,t}$ and $F_t$ be non-empty compact subsets of $\R^2$ such that $F_{k,t}\to F_t$ in the Hausdorff metric. If $\lambda_k( \bigcup_{t=1}^p F_{k,t} ) \to 1$ then $\spt\lambda\subset\bigcup_{t=1}^p F_t$.
\end{lemma}
\begin{proof}
    Write $F:=\bigcup_{t=1}^p F_t$. For every $\delta>0$, $\bigcup_{t=1}^p F_{k,t}\subset\mathcal{N}_\delta(F)$ for all large $k$ and hence $\lambda_k(\mathcal{N}_\delta(F))\to1$. Since $\mathcal{N}_\delta(F)$ is compact,
    \[
        \lambda(\mathcal{N}_\delta(F)) \geq \limsup_{k\to\infty} \lambda_k(\mathcal{N}_\delta(F))=1.
    \]
    Letting $\delta\downarrow 0$, we obtain $\lambda(F)=1$. Since $F$ is closed, $\spt\lambda\subset F$.
\end{proof}

\subsection{Coding and the geometry of approximate squares}

Let $\Sigma:=\Gamma^{\N}$ and let $\nu_{\p}$ be the Bernoulli measure on $\Sigma$ with weights $\p=(p_{i,j})_{(i,j)\in\Gamma}$. The coding map $\Pi\colon\Sigma\to K$ is defined by
\[
    \Pi(\omega) = \Big( \sum_{t=1}^{\infty}\frac{i_t}{n^t}, \sum_{t=1}^{\infty}\frac{j_t}{m^t} \Big), \quad \omega=(i_1,j_1)(i_2,j_2)\cdots\in\Sigma.
\]
Then $\mu_{\p}=\Pi\nu_{\p}$. From now on, we keep $\p$ fixed and omit this subscript from both measures. Let $\Gamma^*:=\bigcup_{k=1}^\infty \Gamma^k$ be the set of finite words. For any $\u=u_1u_2\cdots u_k\in \Gamma^*$, let 
\[
    [\u]:=\{u'_1u'_2\cdots\in \Sigma: u'_t=u_t \text{ for } t\leq k\}
\]
be the corresponding cylinder and let $\phi_{\u}:=\phi_{u_1}\cdots\phi_{u_k}$. Here we slightly abuse the notation $[\cdot]$; the meaning is clear from the context.
By definition, $\Pi(\u\omega)=\phi_{\u}(\Pi\omega)$ for all $\u\in \Gamma^*$ and $\omega\in \Sigma$.

For $\ell\geq 1$, set $\widehat{\ell}:=\lceil\ell\log_m n\rceil$. Equivalently, $m^{-\widehat{\ell}}\leq n^{-\ell}<m^{-\widehat{\ell}+1}$. So $m^{-\widehat{\ell}}$ and $n^{-\ell}$ are uniformly comparable. Moreover, $\hl -\ell\to\infty$ as $\ell\to\infty$. 
For $\omega=(i_t,j_t)_t\in\Sigma$, define the \emph{symbolic approximate square}
\[
    C_\ell(\omega) := \left\{ (u_t,v_t)_{t}\in\Sigma:
    \begin{array}{ll}
        (u_t,v_t)=(i_t,j_t)&\text{for }t\leq\ell,\\
        v_t=j_t&\text{for }\ell<t\leq \hl
    \end{array} \right\},
\]
and put $Q_\ell(\omega):=\Pi(C_\ell(\omega))$. Thus $C_\ell(\omega)$ fixes the first $\ell$ horizontal digits and the first $\widehat{\ell}$ vertical digits. We refer to these sets as \emph{approximate squares of level $\ell$}, with the word \emph{symbolic} distinguishing $C_\ell(\omega)$ from its image $Q_\ell(\omega)$. Their symbolic masses are
\[
    \nu(C_\ell(\omega)) = \prod_{t=1}^{\ell} p_{i_t,j_t}\prod_{t=\ell+1}^{\hl} q_{j_t}.
\]
The partitions of $\Sigma$ with atoms $C_\ell(\omega)$ refine one another and generate the Borel $\sigma$-algebra of $\Sigma$.

For a finite or infinite word $\a=(a_t)$ over $[n]$, write $P_n(\a) := \sum_t a_tn^{-t}$, where the sum is taken over all its coordinates; $P_m$ is defined analogously.  Following~\cite{AH19}, for $\b=b_1\cdots b_q\in J^q$ we set
\begin{equation*} 
    H(\b) := \bigcup_{\substack{\a=a_1\cdots a_q \\ a_t\in\Gamma_{b_t}}} \left[ \begin{pmatrix} n^{-q} & 0 \\ 0 & 1 \end{pmatrix} K + \begin{pmatrix} P_n(\a) \\ 0 \end{pmatrix} \right] \subset [0,1]^2.
\end{equation*}
Roughly speaking, $H(\b)$ records the horizontal structure of $K$ along the row indexed by $\b$, while leaving the vertical coordinate unscaled. For $\eta=(j_t)_{t}\in J^{\N}$, let
\begin{equation}\label{eq:Ceta-Heta}
    C(\eta) := \Big\{ \sum_{t=1}^{\infty}\frac{a_t}{n^t}: a_t\in\Gamma_{j_t} \text{ for all $t$} \Big\} \quad\text{and}\quad H(\eta) := C(\eta)\times\pi_2(K).
\end{equation}
The set $C(\eta)$ consists of the horizontal coordinates corresponding to the row address $\eta$. It is exactly the full horizontal fibre at height $P_m(\eta)$ whenever $P_m(\eta)$ admits a unique base-$m$ expansion. The next lemma identifies $H(\eta)$ as the limit obtained from long finite row words converging coordinatewise to $\eta$.

\begin{lemma}\label{lem:fibreconverge}
    Let $\ell_k\to\infty$ and let $\b_k\in J^{\ell_k}$. If $\b_k\to\eta\in J^{\N}$ coordinatewise, then $H(\b_k)\to H(\eta)$ in the Hausdorff metric.
\end{lemma}
\begin{proof}
    This is the varying-word form of~\cite[Lemma 7.1]{AH19}.
    Write  $\eta=j_1j_2\cdots$ and $\b_k=b_{k,1}\cdots b_{k,\ell_k}$. A point $z\in H(\b_k)$ has the form 
    \[
        z = (P_n(a_1\cdots a_{\ell_k})+n^{-\ell_k}x, y),
    \]
    where $a_t\in\Gamma_{b_{k,t}}$ for $1\leq t\leq \ell_k$ and $(x,y)\in K$.
    
    Let $d_k$ be the length of the longest common prefix of $\b_k$ and $\eta$. By assumption, $d_k\to\infty$. Fix $k\geq 1$. Pick an infinite word $\a'=a'_1a'_2\cdots$ by setting $a'_t=a_t$ for $1\leq t\leq d_k$ and choosing for $t>d_k$ an arbitrary digit $a'_t\in\Gamma_{j_t}$. This is possible because $j_t\in J$. Then the point $z'=(P_n(\a'),y)\in H(\eta)$. The two points share the same second coordinate, and their first coordinates lie in the same interval $P_n(a_1\cdots a_{d_k})+[0,n^{-d_k}]$. Thus $|z-z'|\leq n^{-d_k}$. Since $z$ was arbitrary, $H(\b_k)\subset\mathcal{N}_{n^{-d_k}}(H(\eta))$.
    The converse direction can be proved analogously.
\end{proof}

To apply this convergence to approximate squares, we identify their rescaled form. The following identity is the counterpart of~\cite[Lemma 7.4]{AH19} for our choice of scales.

\begin{lemma}\label{lem:local-form}
    Let $\omega=(i_t,j_t)_t\in\Sigma$. For every $\ell\geq 1$ and $z\in Q_\ell(\omega)$, there is $z'\in[-1,0]^2$ such that 
    \[
        n^\ell(Q_\ell(\omega)-z) = 
        \begin{pmatrix}
            1 & 0 \\ 0 & m^{s_\ell}
        \end{pmatrix}
        (H(\b)+z'),
    \]
    where $s_\ell:=\ell\log_m n-\hl \in(-1,0]$ and $\b=j_{\ell+1}\cdots j_{\hl}$.
\end{lemma}
\begin{proof}
    Let $\ell\geq 1$. By definition, 
    \begin{align}
        Q_\ell(\omega) &= 
        \begin{pmatrix}
            P_n(i_1\cdots i_{\ell}) \\ P_m(j_1 \cdots j_{\hl})
        \end{pmatrix}
        + 
        \bigcup_{\substack{\i=i_{\ell+1}\cdots i_{\hl} \\ i_t\in\Gamma_{j_t}}}
        \left(
        \begin{pmatrix}
            n^{-\ell}P_n(\i) \\ 0
        \end{pmatrix}
        +
        \begin{pmatrix}
            n^{-\hl} & 0 \\ 0 & m^{-\hl}
        \end{pmatrix}K
        \right) \notag \\
        &= 
        \begin{pmatrix}
            P_n(i_1\cdots i_{\ell}) \\ P_m(j_1 \cdots j_{\hl})
        \end{pmatrix}
        +
        \begin{pmatrix}
            n^{-\ell} & 0 \\ 0 & m^{-\hl}
        \end{pmatrix}
        H(\b). \label{eq:Qlomega}
    \end{align}
    So for any $z=(x,y)\in Q_\ell(\omega)$,
    \[
        n^{\ell}(Q_\ell(\omega)-z) = 
        \begin{pmatrix}
            1 & 0 \\ 0 & n^{\ell}m^{-\hl}
        \end{pmatrix}
        H(\b) + n^{\ell}
        \left(
        \begin{pmatrix}
            P_n(i_1\cdots i_{\ell})-x \\ P_m(j_1\cdots j_{\hl})-y
        \end{pmatrix}
        \right).
    \]
    Since $m^{s_{\ell}}=n^{\ell}m^{-\hl}$, the desired identity holds if we take 
    \[
        z' = 
        \begin{pmatrix}
            n^{\ell}(P_n(i_1\cdots i_{\ell})-x) \\ m^{\hl}(P_m(j_1\cdots j_{\hl})-y)
        \end{pmatrix}.
    \]
    Since $z\in Q_\ell(\omega)$, it follows from~\eqref{eq:Qlomega} that $z'\in-H(\b)\subset[-1,0]^2$.
\end{proof}

\subsection{From symbolic density to a limiting cover}

We shall repeatedly choose density points with respect to cylinders and symbolic approximate squares. The following standard consequence of the martingale convergence theorem captures the desired density property.

\begin{lemma}\label{lem:density}
    Let $A\subset\Sigma$ be measurable. Then for $\nu$-almost every $\omega\in A$, 
    \[
        \lim_{\ell\to\infty} \frac{\nu(A\cap[\omega|_{\ell}])}{\nu([\omega|_{\ell}])} = 1 \quad\text{and}\quad \lim_{\ell\to\infty} \frac{\nu(A\cap C_\ell(\omega))}{\nu(C_\ell(\omega))} = 1.
    \]
\end{lemma}

Here and later $\omega|_\ell$ denotes the prefix of $\omega$ of length $\ell$.

\begin{proof}
    For the second assertion, let $\mathcal{F}_\ell$ be the finite $\sigma$-algebra generated by the level-$\ell$ symbolic approximate squares. These $\sigma$-algebras increase with $\ell$ and generate the Borel $\sigma$-algebra of $\Sigma$. Since every approximate square has positive $\nu$-measure, 
    \[
        \mathbb{E}_\nu(\mathbbm{1}_A\mid\mathcal{F}_\ell)(\omega) = \frac{\nu(A\cap C_\ell(\omega))}{\nu(C_\ell(\omega))}, \quad \forall \omega\in\Sigma.
    \]
    The increasing martingale theorem therefore gives convergence to $\mathbbm{1}_A(\omega)$ for $\nu$-almost every $\omega$, which proves the second assertion. The first assertion follows in exactly the same way from the usual cylinder filtration.
\end{proof}

\begin{lemma}\label{lem:fullsptlimit}
    Let $\ell_k\to\infty$ and $\omega_k\in\Sigma$. Choose $z_k\in Q_{\ell_k}(\omega_k)$. Define $T_k(z)=n^{\ell_k}(z-z_k)$ for $z\in\R^2$, and define $\lambda_k:=T_k\Pi(\frac{\nu|_{C_{\ell_k}(\omega_k)}}{\nu(C_{\ell_k}(\omega_k))})$. If $T_k(Q_{\ell_k}(\omega_k))\to F$ in the Hausdorff metric and $\lambda_k$ converges weakly to $\lambda$ along a subsequence, then $\spt\lambda=F$.
\end{lemma}
\begin{proof}
    For simplicity, we abbreviate $C_k:=C_{\ell_k}(\omega_k)$ and $Q_k:=Q_{\ell_k}(\omega_k)$. For $x\in\spt\lambda$, if $x\notin F$ then there exists $\delta>0$ such that $B(x,2\delta)\cap F=\varnothing$. Then $B(x,\delta)\cap T_k(Q_{k}) = \varnothing$ for all large $k$ and hence $\lambda_k(B(x,\delta))=0$. It follows that
    \[
        \lambda(B(x,\delta)) \leq \liminf_{k\to\infty}\lambda_k(B(x,\delta)) = 0,
    \]
    contrary to $x\in\spt\lambda$. So $\spt\lambda\subset F$.

    For the reverse inclusion, fix $x\in F$ and $\rho>0$. Since $T_k(Q_k)\to F$ in the Hausdorff metric, we can find $x_k\in T_k(Q_k)$ with $x_k\to x$. Pick $\tau_k\in C_{k}$ such that $T_k(\Pi(\tau_k))=x_k$. Take a large integer $b\geq1$ with $\sqrt2\,m^{-b}<\rho/2$. Let $R_k\subset C_{k}$ consist of sequences which also agree with $\tau_k$ in the horizontal digits at levels $\ell_k+1,\ldots,\ell_k+b$ and in the vertical digits at levels $\widehat{\ell_k}+1,\ldots,\widehat{\ell_k}+b$. Since $\widehat{\ell_k}-\ell_k\geq b$ for all large $k$, we have 
    \begin{equation}\label{eq:disjointlevels}
        \{\ell_k+1,\ldots,\ell_k+b\} \cap \{\widehat{\ell_k}+1,\ldots,\widehat{\ell_k}+b\} = \varnothing.
    \end{equation}
    For $\tau\in R_k$, the first $\ell_k+b$ horizontal digits and the first $\widehat{\ell_k}+b$ vertical digits agree with those of $\tau_k$. So
    \[
        |\pi_1\Pi(\tau)-\pi_1\Pi(\tau_k)| \leq n^{-\ell_k-b} \quad\text{and}\quad |\pi_2\Pi(\tau)-\pi_2\Pi(\tau_k)| \leq m^{-\widehat{\ell_k}-b}.
    \]
    Then 
    \[
        n^{\ell_k}|\pi_1\Pi(\tau)-\pi_1\Pi(\tau_k)|\leq n^{-b}<\rho/2 \quad\text{and}\quad n^{\ell_k}|\pi_2\Pi(\tau)-\pi_2\Pi(\tau_k)| \leq m^{-b} < \rho/2,
    \]
    that is, $T_k(\Pi(R_k))\subset B(x_k,\sqrt{2}\rho/2)$. Since $x_k\to x$, $T_k(\Pi(R_k))\subset \overline{B(x,\rho)}$ for all large $k$.

    Recall that $q_j :=\sum_{i\in\Gamma_j}p_{i,j}$ is the total weight of the $j$-th row. Set
    \[
        c := \min\Big\{ \min_{(i,j)\in\Gamma}\frac{p_{i,j}}{q_j}, \min_{j\in J}q_j  \Big\}>0.
    \]
    Write $\tau_k=(i_{k,t},j_{k,t})_{t}$. By~\eqref{eq:disjointlevels}, we may multiply the conditional probabilities to obtain
    \[
        \frac{\nu(R_k)}{\nu(C_{k})} = \prod_{t={\ell_k+1}}^{\ell_k+b} \frac{p_{i_{k,t},j_{k,t}}}{q_{j_{k,t}}} \cdot \prod_{t={\widehat{\ell_k}+1}}^{\widehat{\ell_k}+b} q_{j_{k,t}} \geq \Big( \min_{(i,j)\in\Gamma}\frac{p_{i,j}}{q_j} \Big)^b\big( \min_{j\in J}q_j \big)^b \geq c^{2b},
    \]
    which is independent of $k$. So $\lambda_k(\overline{B(x,\rho)})\geq c^{2b}$ for all large $k$. It follows that
    \[
        \lambda\bigl( \overline{B(x,\rho)} \bigr) \geq \limsup_{k\to\infty}\lambda_k\bigl( \overline {B(x,\rho)} \bigr) \geq c^{2b}>0.
    \]
    Since $\rho>0$ was arbitrary, $x\in\spt\lambda$, which proves the reverse inclusion.
\end{proof}

Fix a similitude $f(z)=rOz+a$ with $r>0$. Pick a large integer $d\geq 1$ such that 
\begin{equation}\label{eq:kappa}
    \kappa := rn^d\geq \sqrt{2}.
\end{equation}
We keep $f$, $d$ and $\kappa$ fixed in Sections~\ref{sec:symbolic-pre} and~\ref{sec:nonobliqueness}. To cover an intersection inside an approximate square of level $\ell$, we use squares of level $\ell-d$. 
The next lemma is of the same type as~\cite[Lemma 7.3]{AH19}.

\begin{lemma}\label{lem:finite-covering}
    For every $\ell\geq d+1$ and $\omega\in\Sigma$, the set $Q_\ell(\omega)\cap f(K)$ is covered by at most $9m$ sets of the form $f(Q_{\ell-d}(\tau))$.
\end{lemma}
\begin{proof}
    Since $f$ is a similitude, it suffices to show that $f^{-1}(Q_{\ell+d}(\omega))\cap K$ is covered by at most $9m$ sets of the form $Q_{\ell}(\tau)$ for all $\ell\geq 1$. By definition, $f^{-1}(Q_{\ell+d}(\omega))$ is contained in a (possibly rotated) rectangle with side lengths $r^{-1}n^{-(\ell+d)}$ and $r^{-1}m^{-(\widehat{\ell+d})}$. By~\eqref{eq:kappa}, 
    \[
        \max\big\{ r^{-1}n^{-(\ell+d)}, r^{-1}m^{-(\widehat{\ell+d})} \big\} \leq r^{-1}n^{-(\ell+d)} = n^{-\ell}/\kappa \leq n^{-\ell} /\sqrt{2}.
    \]
	Thus $f^{-1}(Q_{\ell+d}(\omega))$ is contained in an axis-parallel square $S_0$ with side length $n^{-\ell}$.
	
	For each $\tau=(i_t,j_t)_t\in \Sigma$, define $R_{\ell}(\tau)$ to be the axis-parallel rectangle
	\[
	    [P_n(i_1\cdots i_{\ell}), P_n(i_1\cdots i_{\ell})+n^{-\ell}] \times \big[P_m(j_1\cdots j_{\widehat{\ell}}), P_m(j_1\cdots j_{\widehat{\ell}})+m^{-\widehat{\ell}}\big].
	\]
	Then $R_{\ell}(\tau)$ contains $Q_{\ell}(\tau)$ and has side lengths $n^{-\ell}$ and $m^{-\widehat{\ell}}$. Note that for $\tau\not=\tau'$, either $R_{\ell}(\tau)=R_{\ell}(\tau')$ or their interiors are disjoint. 
	
	Choose $\tau_1,\ldots,\tau_N$ so that $R_\ell(\tau_1),\ldots,R_\ell(\tau_N)$ are pairwise distinct and 
    \[
        \{R_\ell(\tau_1),\ldots,R_\ell(\tau_N)\} = \{R_\ell(\tau):\tau\in \Sigma \text{ and } R_\ell(\tau)\cap S_0\neq\varnothing \}.
    \]
    Then $S_0\cap K\subset \bigcup_{t=1}^N Q_\ell(\tau_t)$. Since each $R_\ell(\tau_t)$ intersects $S_0$ and has both side lengths at most $n^{-\ell}$, $\bigcup_{t=1}^N R_\ell(\tau_t)$ is contained in an axis-parallel square with side length $3n^{-\ell}$. So
	\[
	  (3n^{-\ell})^2\geq \mathcal{L}^2\Big(\bigcup_{t=1}^N R_\ell(\tau_t)\Big)=\sum_{t=1}^N \mathcal{L}^2(R_{\ell}(\tau_t))= N n^{-\ell}m^{-\widehat{\ell}}> Nn^{-2\ell}/m
	\]
	and thus $N< 9m$.
\end{proof}

We now choose density approximate squares whose row words begin with increasingly long blocks from a prescribed row $j_0$. Their limit is the fixed product set $H(j_0^\infty)$, up to a diagonal transformation and a translation. We use this flexibility in the first two cases of the proof of Theorem~\ref{thm:nonoblique}.

\begin{proposition}\label{prop:prescribed-cover}
    Let $j_0\in J$. If $\mu(K\cap f(K))>0$, then there exist $1\leq p\leq 9m$, $\alpha_0,\alpha_1\in[m^{-1},1]$, $\eta_1,\ldots,\eta_p\in J^{\N}$, and $c_1,\ldots,c_p\in\R^2$ such that 
    \begin{equation*}
        D_0H(j_0^\infty) \subset \bigcup_{t=1}^p \bigl( \kappa OD_1H(\eta_t)+c_t \bigr),
    \end{equation*}
    where $D_i=\operatorname{diag}(1,\alpha_i)$ for $i=0,1$.
\end{proposition}
\begin{proof}
    Let $A:=\Pi^{-1}(K\cap f(K))$. So $\nu(A)>0$. Lemma~\ref{lem:density} gives us an ordinary cylinder density point $\omega^*\in A$. Choose a fast increasing sequence $\ell_N\to\infty$ such that $\ell_N>d$, $\widehat{\ell_N}-\ell_N\geq N$ and
    \[
        \nu([\omega^*|_{\ell_N}]\setminus A) < \frac{q_{j_0}^{\,N}}{N}\nu([\omega^*|_{\ell_N}]).
    \]
    Let $E_N\subset[\omega^*|_{\ell_N}]$ consist of sequences whose vertical digits at levels $\ell_N+1,\ldots,\ell_N+N$ are all $j_0$. By the product structure, we have 
    \[
        \nu(E_N) = q_{j_0}^{\,N}\nu([\omega^*|_{\ell_N}]) \quad\text{and}\quad \nu(E_N\setminus A) \leq \nu([\omega^*|_{\ell_N}]\setminus A) < \frac{1}{N}\nu(E_N).
    \]
    So the density of $A$ in $E_N$ is greater than $1-1/N$.

    Since $\widehat{\ell_N}\geq \ell_N+N$, the set $E_N$ is a finite disjoint union of level-$\ell_N$ symbolic approximate squares. Note that
    \[
        1-\frac{1}{N} < \ \frac{\nu(A\cap E_N)}{\nu(E_N)} = \sum_{\substack{C: \text{ a level-$\ell_N$ atom} \\ C\subset E_N}}
        \frac{\nu(C)}{\nu(E_N)}\cdot\frac{\nu(A\cap C)}{\nu(C)}.
    \]
    Therefore at least one atom, denoted $C_N$, satisfies
    \begin{equation}\label{eq:atom-density}
        \frac{\nu(A\cap C_N)}{\nu(C_N)} > 1-\frac{1}{N}.
    \end{equation}
    If $\b_N$ denotes the row word fixed by $C_N$ between levels $\ell_N+1$ and $\widehat{\ell_N}$, then $\b_N$ begins with $j_0^N$ and thus $\b_N\to j_0^\infty$ coordinatewise.

    Let $Q_N:=\Pi(C_N)$ and choose $x_N\in Q_N$. For $\ell\geq 1$, write $\widetilde{D}_\ell= \diag(1,m^{s_\ell})$, where $s_\ell := \ell\log_m n-\widehat{\ell}\in (-1,0]$. By Lemma~\ref{lem:local-form}, there are $z_N\in[-1,0]^2$ such that
    \[
        n^{\ell_N}(Q_N-x_N)=\widetilde D_{\ell_N}H(\b_N)+z_N.
    \]
    Passing to a common subsequence, we may assume that $\{\widetilde{D}_{\ell_N}\}$, $\{\widetilde{D}_{\ell_N-d}\}$ and $z_N$ are all convergent, say
    \[
        \widetilde{D}_{\ell_N}\to D_0,\, \widetilde{D}_{\ell_N-d}\to D_1,\, \text{ and } z_N\to z_0.
    \]
    By Lemmas~\ref{lem:fibreconverge} and~\ref{lem:hd-continuity},
    \begin{equation}\label{eq:pre-limit}
        n^{\ell_N}(Q_N-x_N) \to D_0H(j_0^\infty)+z_0.
    \end{equation}

    We next take limits on the cover side. By Lemma~\ref{lem:finite-covering}, we can choose at most $9m$ level-$(\ell_N-d)$ approximate squares such that their images under $f$ cover $Q_N\cap f(K)$, with each image meeting this intersection. After passing to a subsequence, the number of these squares is a constant $1\leq p\leq 9m$, say $E_{N,1},\ldots,E_{N,p}$. For each $t$, choose $y_{N,t}\in E_{N,t}$ with $f(y_{N,t})\in Q_N$. If $\b_{N,t}$ denotes the row word fixed by $E_{N,t}$ between levels $\ell_N-d+1$ and $\widehat{\ell_N-d}$, then by Lemma~\ref{lem:local-form} there exist $z_{N,t}\in[-1,0]^2$ such that
    \[
        n^{\ell_N-d}(E_{N,t}-y_{N,t}) = \widetilde{D}_{\ell_N-d}H(\b_{N,t})+z_{N,t}.
    \]
    Recalling~\eqref{eq:kappa}, we have
    \begin{align*}
        n^{\ell_N}(f(E_{N,t})-x_N) &= n^{\ell_N}(f(E_{N,t})-f(y_{N,t}))+n^{\ell_N}(f(y_{N,t})-x_N) \\
        &= \kappa O n^{\ell_N-d}(E_{N,t}-y_{N,t}) + n^{\ell_N}(f(y_{N,t})-x_N) \\
        &= \kappa O\widetilde{D}_{\ell_N-d}H(\b_{N,t}) + c_{N,t},
    \end{align*}
    where the two translations are merged into $c_{N,t} := \kappa Oz_{N,t}+n^{\ell_N}(f(y_{N,t})-x_N)$. 
    Both $f(y_{N,t})$ and $x_N$ lie in $Q_N$, whose diameter is at most $\sqrt2\,n^{-\ell_N}$. So $|c_{N,t}| \leq \sqrt{2}+\sqrt{2}\kappa$.

    The words $\b_{N,t}$ have lengths $\widehat{\ell_N-d}-(\ell_N-d)\to\infty$. Since $J$ is finite, a diagonal subsequence argument makes every fixed initial word of each $\b_{N,t}$ eventually constant. The translations are bounded, and there are only $p\leq 9m$ indices $t$. We may therefore take one further subsequence on which $\{\b_{N,t}\}$ and $\{c_{N,t}\}$ are both convergent, say $\b_{N,t}\to\eta_t\in J^{\N}$ coordinatewise and $c_{N,t}\to c'_t$ for each $1\leq t\leq p$. Then
    \[
        n^{\ell_N}(f(E_{N,t})-x_N) \to \kappa OD_1H(\eta_t)+c'_t, \quad \forall 1\leq t\leq p.
    \]
    Let $T_N(z):=n^{\ell_N}(z-x_N)$ and let $\lambda_N := T_N\Pi(\frac{\nu|_{C_N}}{\nu(C_N)})$. 
    By~\eqref{eq:pre-limit}, the supports of $\lambda_N$ are contained in a fixed compact set. Passing to a further subsequence, assume that $\lambda_N$ converges weakly to some $\lambda$. By Lemma~\ref{lem:fullsptlimit} and~\eqref{eq:pre-limit}, $\spt\lambda=D_0H(j_0^\infty)+z_0$.

    Since $Q_N\cap f(K) \subset \bigcup_{t=1}^p f(E_{N,t})$ and $Q_N=\Pi(C_N)$, 
    \[
      C_N \cap \Pi^{-1}\Big( \bigcup_{t=1}^p f(E_{N,t})\Big)\supset C_N\cap \Pi^{-1}(Q_N\cap f(K)) =  C_N\cap A.
    \]
    It then follows from~\eqref{eq:atom-density} that
    \[
        \lambda_N\Big( \bigcup_{t=1}^p T_N(f(E_{N,t})) \Big) = \frac{1}{\nu(C_N)}\nu\Big( C_N \cap \Pi^{-1}\Big( \bigcup_{t=1}^p f(E_{N,t}) \Big) \Big) \geq \frac{\nu(A\cap C_N)}{\nu(C_N)} \to 1.
    \]
    Moreover, we have seen that $T_N(f(E_{N,t}))\to\kappa OD_1H(\eta_t)+c'_t$ for all $1\leq t\leq p$. Lemma~\ref{lem:finite-cover-limit} therefore yields
    \[
        D_0H(j_0^\infty)+z_0 = \spt\lambda \subset \bigcup_{t=1}^p \bigl( \kappa OD_1H(\eta_t)+c'_t \bigr).
    \]
    Subtracting $z_0$ proves the proposition with $c_t=c'_t-z_0$.
\end{proof}

\section{Proof of Theorem~\ref{thm:nonoblique}}
\label{sec:nonobliqueness}

\subsection{Vacant rows and the case without full rows}

We first treat the cases in which there is a vacant row or every row is non-empty but none is full. For $q\geq2$ and a non-empty digit set $\D\subset[q]$,  the deleted-digit set with base $q$ and digits $\D$ is defined by
\begin{equation}\label{eq:deleted-digit}
    E(q,\D) := \Big\{ \sum_{t=1}^{\infty}d_tq^{-t}:d_t\in\D \Big\}.
\end{equation}
A point $x$ of a compact set $E\subset\R$ is \emph{left isolated} if
$(x-\varepsilon,x)\cap E=\varnothing$ for some $\varepsilon>0$;
right isolation is defined similarly.

\begin{lemma}\label{lem:isolated-points}
    Let $q\geq 3$ and let $\D\subset[q]$ satisfy $2\leq\#\D<q$. If a contracting similitude $h$ sends $E(q,\D)$ into itself, then its image contains both a left isolated point and a right isolated point of $E(q,\D)$.
\end{lemma}
\begin{proof}
    Consider the Bedford--McMullen carpet
    \[
        K' = K(q,2,\D\times\{0,1\}) = E(q,\D)\times[0,1].
    \]
    Every horizontal slice of $K'$ equals $E(q,\D)$. Moreover, writing $\alpha:=\frac{\log\#\D}{\log q}$, we have $0<\h^\alpha(E(q,\D))<\infty$. Hence~\cite[Corollary 3.4]{Xiao26} applies to the self-embedding $h(E(q,\D))\subset E(q,\D)$. Then we obtain a left and a right isolated point of $E(q,\D)$ whose images under $h$ are also one-sided isolated points (with left and right interchanged if $h$ is orientation-reversing), as desired.
\end{proof}

\begin{proposition}\label{prop:first-cases}
    Theorem~\ref{thm:nonoblique} holds if $J\neq[m]$ or if $J=[m]$ and $J_*=\varnothing$.
\end{proposition}
\begin{proof}
    Assume that $\mu(K\cap f(K))>0$ but $f$ is oblique. Choose $j_0\in J$ with $\#\Gamma_{j_0}$ maximal. Applying Proposition~\ref{prop:prescribed-cover} to $j_0$, we get 
    \begin{equation}\label{eq:finitecover}
        D_0H(j_0^\infty) \subset \bigcup_{t=1}^p G_t,
    \end{equation}
    where
    \[
        G_t := \kappa OD_1H(\eta_t)+c_t, \quad 1\leq t\leq p,
    \]
    and $p, D_0, D_1, \eta_t, c_t$ are as described there. Recall from~\ref{eq:Ceta-Heta} that $H(\eta_t)=C(\eta_t)\times\pi_2(K)$. 

    \medskip
    \noindent\emph{Case 1: $J=[m]$ and $J_*=\varnothing$.}
    For $\eta=(j_t)_t\in[m]^{\N}$, the set $C(\eta)$ is covered by the intervals
    \[
        [P_n(a_1\cdots a_N),P_n(a_1\cdots a_N)+n^{-N}], \quad a_t\in\Gamma_{j_t} \text{ for } 1\leq t\leq N.
    \]
    There are $\prod_{t=1}^N\#\Gamma_{j_t}$ such intervals. Since no row is full,
    \[
        \mathcal L^1(C(\eta))
        \leq n^{-N}\prod_{t=1}^N\#\Gamma_{j_t}
        \leq\left(\frac{n-1}{n}\right)^N\to0, \quad N\to\infty.
    \]
    In particular, $C(\eta)$ does not contain any nondegenerate interval. Thus no $G_t$ can contain a nondegenerate vertical segment. Indeed, applying the inverse of $z\mapsto\kappa OD_1z+c_t$ to such a segment would give an oblique segment in $C(\eta_t)\times\pi_2(K)$ because $O$ is oblique and $D_1$ is diagonal. Its projection under $\pi_1$ would therefore be an interval in $C(\eta_t)$, a contradiction.

    On the other hand, since $J=[m]$, we have $\pi_2(K)=[0,1]$ and $H(j_0^\infty)=C(j_0^\infty)\times [0,1]$. Choose an arbitrary $x_*\in C(j_0^\infty)$. Then the vertical segment
    \[
        S := D_0(\{x_*\}\times[0,1]) \subset D_0H(j_0^\infty)\subset \bigcup_{t=1}^p G_t.
    \]
    So the closed sets $S\cap G_t$ cover $S$. By Baire's theorem, some $G_t$ contains an interior part of $S$, which contradicts our previous observation.
    
    \medskip
    \noindent\emph{Case 2: $J\neq[m]$.}
    Write $E:=E(n,\Gamma_{j_0})$ and $F:=E(m,J)$. Then $H(j_0^\infty)=E\times F$ and $H(\eta_t)=C(\eta_t)\times F$ for $1\leq t\leq p$. So~\eqref{eq:finitecover} becomes 
    \begin{equation}\label{eq:finitecover2}
        D_0H(j_0^\infty) = D_0(E\times F) \subset \bigcup_{t=1}^p G_t.
    \end{equation}

    If $\#J=1$, then $F$ is a singleton. Since $\#\Gamma>1$, we have $\#\Gamma_{j_0}>1$ and $E$ is infinite. So $H(j_0^\infty)$ is an infinite subset of a horizontal line, whereas each $G_t$ lies on an oblique line which meets that horizontal line at most once. This contradicts~\eqref{eq:finitecover}. If $\#J\geq2$ but $\#\Gamma_{j_0}=1$, our choice of $j_0$ implies that every non-empty row has exactly one digit. Now $E$ and all $C(\eta_t)$ are singletons, while $F$ is infinite. Then the target lies on a vertical line, whereas each $G_t$ lies on an oblique line. The same argument leads to a contradiction.

    Suppose $2\leq\#J<m$ and $\#\Gamma_{j_0}\geq2$. Applying Baire's theorem to~\eqref{eq:finitecover2}, we see that some $G_t$ contains a non-empty relatively open subset of $D_0(E\times F)$. Since $D_0$ is a homeomorphism, there is a non-empty relatively open set $U\subset E\times F$ with $D_0(U)\subset G_t$.
    For a large $k\geq1$, there is a cylinder map $\phi$ (of the product IFS $E\times F$) with linear part $\operatorname{diag}(n^{-k},m^{-k})$ such that
    \[
        D_0\phi(E\times F) \subset G_t = \kappa OD_1(C(\eta_t)\times F)+c_t.
    \]
    Writing $g: z\mapsto \kappa^{-1}D_1^{-1}O^{-1}(D_0\phi(z)-c_t)$, we then have 
    \begin{equation}\label{eq:g-product-inclusion}
        g(E\times F)\subset C(\eta_t)\times F.
    \end{equation}

    Let $D_i:=\operatorname{diag}(1,\alpha_i)$ for $i=0,1$ and $O=(o_{ij})$. The second coordinate of $g$ is
    \[
        g_2(x,y)=ux+vy+b,
    \]
    where $u:=\frac{o_{12}}{\kappa \alpha_1n^k}$ and $v:=\frac{o_{22}\alpha_0}{\kappa \alpha_1m^k}$. Obliqueness gives $u,v\neq0$. Also, $\alpha_0\leq1$, $\alpha_1\geq m^{-1}$, and $k\geq1$, so
    \[
        0 < |v| \leq \kappa^{-1}m^{1-k} \leq \kappa^{-1} < 1.
    \]
    Then for each fixed $x\in E$, the map $y\mapsto g_2(x,y)$ is a contracting similitude sending $F$ into itself.

    Take $x_0:=\min E$. Since $2\leq\#J<m$, Lemma~\ref{lem:isolated-points} applies. Choose $y_0\in F$ such that $g_2(x_0,y_0)$ is right isolated in $F$ when $u>0$, and left isolated when $u<0$. There are points of $E$ converging to $x_0$ from the right, say $x_q\downarrow x_0$. 
    But
    \[
        g_2(x_q,y_0)=g_2(x_0,y_0)+u(x_q-x_0)
    \]
    approaches $g_2(x_0,y_0)$ from its isolated side. For all sufficiently large $q$, these points lie in the corresponding adjacent gap of $F$, contradicting~\eqref{eq:g-product-inclusion}.
\end{proof}

\subsection{Full rows and exceptional vertical coordinates}

It remains to consider carpets with every row non-empty and at least one full row. A horizontal fibre can now contain an interval, so the previous vertical-segment argument requires a suitable choice of scales. We first exclude certain heights via a measure estimate. Write $\T:=\R/\Z$.

\begin{lemma}\label{lem:discrete}
    For every non-empty open arc $I\subset\T$, there is $\ell\geq 1$ such that, for all $\theta\in\T$ and all $\alpha\in\R$ with $1\leq |\alpha|<n$, $\theta+i\alpha n^{-\ell}\in I\pmod 1$ for some $i\in\{0,\ldots,n^{\ell}-1\}$.
\end{lemma}
\begin{proof}
    We omit the proof, which is an elementary discretization argument.
\end{proof}

The existence of a full row provides a useful uniform spreading property.

\begin{lemma}\label{lem:spreading}
    Assume that $J_*\neq\varnothing$. For every non-empty open arc $I\subset\T$, there is $c_I>0$ such that
    \begin{equation*}
        \mu\bigl\{(x,y)\in K: \alpha x+\beta y+\gamma\in I\pmod\Z \bigr\} \geq c_I
    \end{equation*}
    for all $\alpha,\beta,\gamma\in\R$ with $|\alpha|\geq1$.
\end{lemma}

The key point is that a full row supplies $n^\ell$ equally spaced horizontal translations for all $\ell\geq 1$, and Lemma~\ref{lem:discrete} forces one of them into $I$ uniformly in the tail.


\begin{proof}
    First suppose $1\leq|\alpha|<n$. Fix any $j_*\in J_*$. Let $\ell=\ell(I)$ be given by Lemma~\ref{lem:discrete} and let $c_I:=\min_{i\in[n]} p_{i,j_*}^\ell>0$.
    For simplicity, write $L(x,y) := \alpha x+\beta y+\gamma \pmod\Z$ and $E := \{z\in K: L(z)\in I\}$.

    Let $\mathcal{W}\subset\Gamma^\ell$ be the collection of words whose vertical digits are all $j_*$. Since $j_*\in J_*$, their horizontal words range over all of $[n]^\ell$. For $w=(i_1,j_*)\cdots(i_\ell,j_*)\in\mathcal{W}$, 
    \[
        \phi_w(x,y) = \Big( n^{-\ell}(x+k), m^{-\ell}y+j_*\sum_{t=1}^\ell m^{-t} \Big),
    \]
    where $k:=\sum_{t=1}^\ell i_tn^{\ell-t}$. As $w$ varies over $\mathcal{W}$, the integer $k$ ranges over $\{0,1,\ldots,n^\ell-1\}$. Thus for each fixed $z=(x,y)\in K$, the values $L\phi_w(z)$ have the form
    \[
        \theta(z) + \alpha k n^{-\ell}, \quad 0\leq k<n^\ell,
    \]
    where $\theta(z):=\alpha n^{-\ell}x+\beta(m^{-\ell}y+j_*\sum_{t=1}^\ell m^{-t})+\gamma$ is independent of $w$. By Lemma~\ref{lem:discrete}, at least one of these values lies in $I$ modulo $\Z$. Since $\phi_w(z)\in K$, we have
    \[
        \sum_{w\in\mathcal{W}} \mathbbm{1}_E(\phi_w(z)) \geq 1, \quad \forall z\in K.
    \]
    
    For every $w\in\mathcal{W}$, $\nu([w])\geq c_I$. By the self-affinity of $\mu$, we have 
    \[
      \mu(E) = \sum_{w\in\Gamma^\ell} \nu([w])\phi_w\mu(E) \geq \sum_{w\in\mathcal{W}} c_I\phi_w\mu(E).      
    \]
    Moreover,
    \[
       \sum_{w\in\mathcal{W}} \phi_w\mu(E)=\sum_{w\in\mathcal{W}} \int_K \mathbbm{1}_E(z) \rd\phi_w\mu(z) = \int_K \sum_{w\in\mathcal{W}} \mathbbm{1}_E(\phi_w(z)) \rd\mu(z) \geq 1.      
    \]
    So the desired estimate holds for $1\leq|\alpha|<n$, uniformly in $\beta$ and $\gamma$.

    When $|\alpha|\geq n$, choose an integer $q\geq1$ with $1\leq|\alpha|n^{-q}<n$. For every $u\in\Gamma^q$, the affine function $L\phi_u$ has first-coordinate coefficient $\alpha n^{-q}$. The case just proved then gives
    \[
        \mu(\phi_u^{-1}(E)) = \mu\{z\in K: L\phi_u(z)\in I\} \geq c_I.
    \]
    By the self-affinity of $\mu$,  
    \[
        \mu(E) = \sum_{u\in\Gamma^q} \nu([u])\mu(\phi_u^{-1}(E)) \geq c_I\sum_{u\in\Gamma^q}\nu([u]) = c_I,
    \]
    which finishes the proof.
\end{proof}

We apply the spreading estimate to heights whose addresses spend asymptotically all the time in full rows. Define
\begin{equation}\label{eq:sparse-set}
    \mathcal{Y} := \Big\{ P_m(\eta)\pmod\Z: \eta=(j_t)_t\in[m]^{\N}, \lim_{N\to\infty} \frac{1}{N} \#\{1\leq t\leq N:j_t\in J_*\}=1 \Big\}.
\end{equation}
If $y\in \mathcal{Y}$ then at least one of its vertical addresses has the frequency property. The set $\mathcal{Y}$ is Borel. To see this, let 
\[
    \mathcal{D}_m:=\{i/m^k: k\geq 0, i\in\mathbb{Z}, 0\leq i< m^k\}
\]
be the set of $m$-adic rationals in $[0,1)$ (equivalently, in $\T$). Outside the countable set $\mathcal{D}_m$, the base-$m$ address is unique and each of its digits is a Borel function of the height (piecewise constant). So the set of heights in $\T\setminus\mathcal{D}_m$ satisfying the frequency condition in~\eqref{eq:sparse-set} is Borel. The remaining part of $\mathcal{Y}$ is countable and hence also Borel.

The next proposition will be applied to $L=\pi_2f^{-1}\pmod\Z$. If $f$ is oblique, the horizontal component of this function is non-zero. This enables us to select a density point $x$ in the intersection such that $L(x)\notin\mathcal{Y}$.

\begin{proposition}\label{prop:sparse-null}
    Assume that $J_*\neq\varnothing$. Let $L(x,y)=ux+vy+b\pmod\Z$ with $u\neq 0$ and let $B:=\{\omega\in\Sigma: L(\Pi\omega)\in \mathcal{Y}\}$. Then $\nu(B)=0$. As a result, 
    \[
        \mu\{z\in K: L(z)\in \mathcal{Y}\}=0.
    \]
\end{proposition}
\begin{proof}
	Suppose $\nu(B)>0$. Since $\#\Gamma<mn$, $J_*\neq[m]$. Choose $j_0\in[m]\setminus J_*$ and a non-empty open interval $I$ whose closure lies in $(j_0/m,(j_0+1)/m)$. Let $c_I$ be the constant from Lemma~\ref{lem:spreading}. By Lemma~\ref{lem:density}, there is a cylinder $[w]\subset\Sigma$ of level $k$ such that $\frac{\nu(B\cap[w])}{\nu([w])} > 1-\frac{c_I}{2}$. Letting $B_w :=\{\xi\in\Sigma: w\xi\in B\}$ be the tail section, we have $L(\Pi(w\xi))\in\mathcal{Y}$ whenever $\xi\in B_w$, and
    \[
        \nu(B_w) = \frac{\nu(B\cap[w])}{\nu([w])} > 1-\frac{c_I}{2}.
    \]
    Define $T_m(y) :=my\pmod\Z$ and 
    \[
        G_N(\xi) := \frac{1}{N}\sum_{t=0}^{N-1} \mathbbm{1}_I\bigl( T_m^tL(\Pi(w\xi)) \bigr), \quad \xi\in\Sigma.
    \]

    For any point $y\in \mathcal{Y}$, say $y=P_m(\eta)$ with $\eta=j_1j_2\cdots$ satisfying the frequency condition in~\eqref{eq:sparse-set}, since $T_m^ty = P_m(j_{t+1}j_{t+2}\cdots) \pmod\Z$, if $T_m^t(y)\in I$ then $j_{t+1}=j_0$. So
    \[
        \lim_{N\to\infty} \frac{1}{N}\sum_{t=0}^{N-1} \mathbbm{1}_I(T_m^t(y)) \leq \lim_{N\to\infty} \frac{1}{N}\#\{0\leq t\leq N-1: j_{t+1}=j_0\} = 0, \quad \forall y\in \mathcal{Y}.
    \]
    As a consequence, $G_N\to 0$ pointwise on $B_w$. Since $0\leq G_N\leq1$, the dominated convergence theorem gives $\int_{B_w} G_N(\xi) \rd\nu(\xi) \to 0$ and thus 
    \begin{equation}\label{eq:limsup-integralofgn}
        \limsup_{N\to\infty} \int_{\Sigma} G_N(\xi) \rd\nu(\xi) = \limsup_{N\to\infty} \int_{\Sigma\setminus B_w} G_N(\xi) \rd\nu(\xi) \leq \nu(\Sigma\setminus B_w) < c_I/2.
    \end{equation}

    On the other hand, 
    recall that $w$ has length $k$. The function $L\phi_w$ has the form
    \[
        L\phi_w(x,y) = un^{-k}x+vm^{-k}y+a_w \pmod\Z
    \]
    for some $a_w\in\R$. Since $u\neq0$, the horizontal coefficient of $T_m^tL\phi_w$ has absolute value $|um^tn^{-k}|\geq 1$ for all sufficiently large $t$. By Lemma~\ref{lem:spreading}, for every such $t$ we have
    \[
        \int_\Sigma \mathbbm{1}_I \bigl( T_m^tL\phi_w(\Pi\xi) \bigr)\rd\nu(\xi) = \nu\{\xi\in\Sigma: T_m^tL\phi_w(\Pi\xi)\in I\} \geq c_I.
    \]
    Since this holds for all large $t$ and $\Pi(w\xi)=\phi_w(\Pi\xi)$, we immediately get 
    \[
        \liminf_{N\to\infty} \int_{\Sigma} G_N(\xi) \rd\nu(\xi) \geq c_I,
    \]
    which contradicts~\eqref{eq:limsup-integralofgn}.
\end{proof}

We also want to exclude those $m$-adic boundary heights. Since their preimages under $L$ are segments,  the following observation (a special case of~\cite[Lemma 2.6]{EKM10}) allows us to discard all of them. We include a proof for completeness.

\begin{lemma}\label{lem:line-null}
    If $K$ is not contained in a line, then every line is $\mu$-null.
\end{lemma}
\begin{proof}
    Assume on the contrary that there is a line $S$ with $\mu(S)>0$. Choose a cylinder density point $\omega\in \Pi^{-1}(S)$. Then 
    \[
        \mu(\phi_{\omega|_k}^{-1}(S)) = \nu(\{\eta\in\Sigma: \Pi(\omega|_k\eta)\in S\}) = \frac{\nu(\Pi^{-1}(S)\cap[\omega|_k])}{\nu([\omega|_k])}\to 1.
    \]
    The line $\phi_{\omega|_k}^{-1}(S)$ can be written as $\phi_{\omega|_k}^{-1}(S) = \{z: \langle e_k,z \rangle = a_k\}$, where $e_k\in S^1$ is a unit normal vector and $a_k\in\R$. 
    Choose $z_k\in\phi_{\omega|_k}^{-1}(S)\cap K$ for all large $k$. Since $K\subset [0,1]^2$, $|a_k|=|\langle e_k, z_k \rangle| \leq\sqrt{2}$. Passing to a subsequence, we may assume that $e_k\to e$ and $a_k\to a$ for some $e\in S^1$ and $a\in [-\sqrt{2}, \sqrt{2}]$. Let $S_\infty :=\{z:\langle e,z\rangle=a\}$. If $z\in \phi_{\omega|_k}^{-1}(S)\cap K$, then
    \begin{align*}
        \dist(z,S_\infty) &= |\langle e,z\rangle-a| \\
        &\leq |\langle e-e_k,z\rangle | + |\langle e_k, z \rangle - a| \\
        &\leq \sqrt{2}|e-e_k|+|a-a_k|.
    \end{align*}
    Thus, for every $\delta>0$, all sufficiently large $k$ satisfy $\phi_{\omega|_k}^{-1}(S)\cap K\subset \mathcal{N}_\delta(S_\infty)$. So
    \[
        \mu(\mathcal{N}_\delta(S_\infty)) \geq \mu(\phi_{\omega|_k}^{-1}(S)) \to 1, \quad k\to\infty.
    \]
    Letting $\delta\downarrow0$, we have $\mu(S_\infty)=1$. Since $S_\infty$ is closed, it contains the support $\spt\mu=K$, a contradiction.
\end{proof}

\subsection{Good scales for vertical addresses}

In the rest of this section, assume that $J=[m]$ and $J_* \neq [m]$. The previous results provide a source point whose height is non-$m$-adic and lies outside $\mathcal{Y}$. We then select scales such that every nearby point has a limiting horizontal fibre with no interval. The lemma below identifies a suitable limiting tail of the original address, and the subsequent proposition shows that this choice persists under small height perturbations.

Let $\sigma$ be the left shift on $[m]^{\N}$. For $\tau\in[m]^{\N}$, write 
\[
    \Xi(\tau) := \{\zeta: \sigma^{n_k}\tau\to\zeta \text{ for some }n_k\to\infty\}.
\]

\begin{lemma}\label{lem:orbit-limit}
    Let $\tau=i_1i_2\cdots\in[m]^{\N}$. If 
    \begin{equation}\label{eq:orbit-limsup-Lem}
        \limsup_{N\to\infty}\frac{1}{N} \#\{1\leq t\leq N:i_t\notin J_*\} > 0,
    \end{equation}
    then there exists $\zeta\in\Xi(\tau)$ with infinitely many coordinates outside $J_*$.
\end{lemma}
\begin{proof}
    By the hypothesis, there are $c>0$ and $N_k\to\infty$ such that
    \begin{equation}\label{eq:orbit-limit}
        \frac1{N_k}\#\{1\leq t\leq N_k:i_t\notin J_*\}\geq c.
    \end{equation}
    Let $\rho_k := \frac{1}{N_k}\sum_{t=0}^{N_k-1}\delta_{\sigma^t\tau}$, where $\delta_{\omega}$ denotes the Dirac measure concentrated at $\omega$. Since $[m]^{\N}$ is compact, there is a subsequence of $\{\rho_k\}$ converging weakly to a probability measure $\rho$. It is easy to see that $\rho$ is $\sigma$-invariant. 

    We claim that $\spt\rho\subset\Xi(\tau)$. For any $\eta\in\spt\rho$, $\rho([\eta|_q])>0$ for every $q\geq 1$. Since $[\eta|_q]$ is clopen, 
    \[
        \lim_{k\to\infty} \rho_k([\eta|_q]) = \rho([\eta|_q]) > 0.
    \]
    Combining this with the definition of $\rho_k$, we see that $\{t: \sigma^t\tau\in[\eta|_q]\}$ is an infinite set. Then there is a strictly increasing sequence $t_q$ with $\sigma^{t_q}\tau\in[\eta|_q]$. These shifted addresses agree with $\eta$ in their first $q$ digits and hence converge to $\eta$. So $\eta\in\Xi(\tau)$.

    Consider $U:=\{\eta=j_1j_2\cdots\in[m]^{\N}:j_1\notin J_*\}$, which is also clopen. By~\eqref{eq:orbit-limit},
    \[
        \rho(U) = \lim_{k\to\infty} \rho_k(U) = \lim_{k\to\infty} \frac{1}{N_k}\#\{1\leq t\leq N_k: i_t\notin J_*\} \geq c>0.
    \]
    By Poincar\'e's recurrence theorem, we can find $\zeta\in U\cap\spt\rho\subset U\cap\Xi(\tau)$ which returns to $U$ infinitely often. Each return $\sigma^t\zeta\in U$ means that the $(t+1)$-th coordinate of $\zeta$ does not belong to $J_*$. This completes the proof.
\end{proof}

\begin{proposition}\label{prop:good-scales}
    Let $\tau=i_1i_2\cdots\in[m]^{\N}$ satisfy \eqref{eq:orbit-limsup-Lem} and $P_m(\tau)\pmod\Z\notin\mathcal{D}_m$. Then there are scales $\ell_k\to\infty$ such that, whenever $\eta_k\in[m]^{\N}$ satisfies
    \begin{equation}\label{eq:nearby-vertical-address}
        |P_m( \eta_k)-P_m(\tau)|\leq m^{-\widehat{\ell_k}+3},
    \end{equation}
    every limit word of $\sigma^{\ell_k}\eta_k$ has infinitely many coordinates outside $J_*$.
\end{proposition}
\begin{proof}
    By Lemma~\ref{lem:orbit-limit}, there are $n_k\to\infty$ and $\zeta\in[m]^{\N}$ such that $\sigma^{n_k}\tau\to\zeta$ and $\zeta$ has infinitely many coordinates outside $J_*$. 
    Note that
    \[
        P_m(\tau) = P_m(\tau|_\ell)+m^{-\ell}P_m(\sigma^\ell\tau), \quad \forall \ell\geq 1.
    \]
    Since $P_m(\tau)\notin\mathcal{D}_m$, it lies in the open interval $(P_m(\tau|_\ell), P_m(\tau|_\ell)+m^{-\ell})$ and
    its distances to the two endpoints are $m^{-\ell}P_m(\sigma^{\ell}\tau)$ and $m^{-\ell}(1-P_m(\sigma^\ell(\tau)))$. 
    Put $\varepsilon_\ell := m^{\ell-\widehat{\ell}+3}$, which tends to zero. So as long as 
    \begin{equation}\label{eq:endpoint-clearance}
        \varepsilon_\ell < \min\{P_m(\sigma^\ell\tau),1-P_m(\sigma^\ell\tau)\},
    \end{equation}
    we always have
    \begin{equation}\label{eq:close-y}
        \forall \eta\in[m]^{\N}, |P_m(\eta)-P_m(\tau)| \leq m^{-\widehat{\ell}+3} \Longrightarrow P_m(\eta)\in \big( P_m(\tau|_\ell), P_m(\tau|_\ell)+m^{-\ell} \big).
    \end{equation}
    Then any such $\eta$ agrees with $\tau$ in its first $\ell$ digits. So
    \begin{equation}\label{eq:tail-y}
        |P_m(\sigma^\ell\eta)-P_m(\sigma^\ell\tau)| = m^{\ell}|P_m(\eta)-P_m(\tau)| \leq m^{\ell-\widehat{\ell}+3} = \varepsilon_\ell.
    \end{equation}
    We will choose scales satisfying~\eqref{eq:endpoint-clearance}.

    \medskip
    \noindent\emph{Case 1: $\zeta$ is neither eventually $0$ nor eventually $m-1$.} In this case, $P_m(\zeta)\notin\mathcal{D}_m$ and hence $\zeta$ is the unique $m$-coding of $P_m(\zeta)$. 
    Take $\ell_k=n_k$. Since $\sigma^{n_k}\tau\to\zeta$,
    \[
        P_m(\sigma^{\ell_k}\tau) \to P_m(\zeta) \in (0,1).
    \]
    So for all large $k$, the values $P_m(\sigma^{\ell_k}\tau)$ stay away from $0$ and $1$ uniformly and thus~\eqref{eq:endpoint-clearance} holds for $\ell=\ell_k$ (recall that $\varepsilon_{\ell_k}\to 0$). Then for any $\{\eta_k\}\subset[m]^{\N}$ satisfying~\eqref{eq:nearby-vertical-address}, we have by~\eqref{eq:close-y} and~\eqref{eq:tail-y} that
    \[
        |P_m(\sigma^{\ell_k}\eta_k) - P_m(\zeta)| \leq |P_m(\sigma^{\ell_k}\eta_k) - P_m(\sigma^{\ell_k}\tau)| + |P_m(\sigma^{\ell_k}\tau)-P_m(\zeta)| \to 0,
    \]
    that is, $P_m(\sigma^{\ell_k}\eta_k) \to P_m(\zeta)$. Since $\zeta$ is the only coding of this point, $\{\sigma^{\ell_k}\eta_k\}$ has only one limit point $\zeta$, which of course has infinitely many coordinates outside $J_*$.

    \medskip
    \noindent\emph{Case 2: $\zeta$ is eventually $0$.}
    In this case, $0\notin J_*$ because $\zeta$ has infinitely many coordinates outside $J_*$. By adding a fixed integer to every $n_k$, we may assume that $\sigma^{n_k}\tau\to 0^\infty$. For all large $k$, let
    \[
        L_k := \max\{p\geq 1: i_{n_k+1}=\cdots=i_{n_k+p}=0\}.
    \]
    In other words, $L_k$ is the length of the initial zero block of $\sigma^{n_k}\tau$, which is a finite value because $\tau$ is not eventually zero (recall that $P_m(\tau)\notin\mathcal{D}_m$). Also, $L_k\to\infty$ since $\sigma^{n_k}\tau\to 0^\infty$. 
    
    Take
    \[
        R_k := \lfloor\sqrt{L_k}\rfloor \quad\text{and}\quad \ell_k = n_k+L_k-R_k.
    \]
    Then $R_k, \ell_k\to\infty$. Moreover, since $\widehat{\ell}-\ell=(\log_m n-1)\ell+O(1)$, $n>m$ and $\ell_k\geq R_k^2-R_k$, we have $\widehat{\ell_k}-\ell_k\geq R_k+5$ for all large $k$. So 
    \[
        \varepsilon_{\ell_k} = m^{\ell_k-\widehat{\ell_k}+3} \leq m^{-R_k-2}.
    \]
    Note that the first $R_k$ digits of $\sigma^{\ell_k}\tau$ are zero, and the next one is non-zero. So
    \begin{equation}\label{eq:tail-near-zero}
        m^{-R_k-1} \leq P_m(\sigma^{\ell_k}\tau) \leq m^{-R_k}.
    \end{equation}
    Then~\eqref{eq:endpoint-clearance} holds for $\ell=\ell_k$.
    Whenever $\eta_k\in[m]^{\N}$ satisfies~\eqref{eq:nearby-vertical-address}, we have by~\eqref{eq:tail-y} that 
    \[
        |P_m(\sigma^{\ell_k}\eta_k) - P_m(\sigma^{\ell_k}\tau)| \leq \varepsilon_{\ell_k} \leq m^{-R_k-2}.
    \]
    Combining this with~\eqref{eq:tail-near-zero},
    \[
        0 \leq P_m(\sigma^{\ell_k}\eta_k) \leq m^{-R_k}+m^{-R_k-2} \to 0.
    \]
    Then for any symbolic limit $\eta'$ of the sequence $\{\sigma^{\ell_k}\eta_k\}$, $P_m(\eta')=0$. So $\eta'=0^\infty$. Since $0\notin J_*$, all its digits lie outside $J_*$.

    \medskip
    \noindent\emph{Case 3: $\zeta$ is eventually $m-1$.}
    This can be resolved in the same manner as Case 2 and we omit the details.
\end{proof}

\subsection{Completion of the full-row case}

To prove the remaining case, we first choose a height to which Proposition~\ref{prop:good-scales} applies. At the resulting scales, all products in the limiting cover have horizontal components of zero length. The argument from Case 1 in Proposition~\ref{prop:first-cases} can therefore be applied again. 

\begin{proposition}\label{prop:fullrow-case}
    Theorem~\ref{thm:nonoblique} holds if $J=[m]$ and $J_*\neq\varnothing$.
\end{proposition}
\begin{proof}
    In this case, $\pi_2(K)=[0,1]$ and it is easy to see that $K$ is not contained in any line. Assume that $\mu(K\cap f(K))>0$ but $f$ is oblique. Recall that $f(z)=rOz+a$. Let $L:=\pi_2f^{-1}\pmod\Z$. Writing $O=(o_{ij})$ and $a=(a_1,a_2)$, we have
    \[
        L(x,y) = r^{-1}\bigl( o_{12}(x-a_1)+o_{22}(y-a_2) \bigr) \pmod\Z.
    \]
    Since $f$ is oblique, $o_{12}\neq0$. By Proposition~\ref{prop:sparse-null}, $\mu\{z\in K:L(z)\in \mathcal{Y}\} = 0$. 
    The set $\{z:L(z)\in\D_m\}$ is a countable union of lines and hence is also $\mu$-null by Lemma~\ref{lem:line-null}.

    Let $A:=\Pi^{-1}(K\cap f(K))$. Then $\nu(A)>0$. By Lemma~\ref{lem:density}, we can choose a word $\omega\in A$ which is an approximate-square density point, that is, 
    \begin{equation}\label{eq:density-final}
        \lim_{\ell\to\infty} \frac{\nu(A\cap C_\ell(\omega))}{\nu(C_\ell(\omega))} = 1.
    \end{equation}
    We may moreover choose $\omega$ so that $y_0:=\Pi(\omega)$ avoids both exceptional sets; that is, 
    \[
        L(\Pi(\omega))\notin\mathcal{D}_m \quad\text{and}\quad L(\Pi(\omega))\notin\mathcal{Y}.
    \]
    Put $x_0:=f^{-1}(y_0)\in K$. So $\pi_2(x_0)=L(y_0)=L(\Pi(\omega))$ is neither $m$-adic nor in $\mathcal{Y}$ modulo $\Z$. Let $\tau=i_1i_2\cdots\in[m]^{\N}$ be the unique base-$m$ expansion of $\pi_2(x_0)$, that is, $\pi_2(x_0)=P_m(\tau)$. Since $P_m(\tau)$ does not belong to $\mathcal{Y}$, \eqref{eq:orbit-limsup-Lem} holds.
    Proposition~\ref{prop:good-scales} therefore applies to $\tau$. Denote its selected levels by $\ell_k\to\infty$.

    Set $\ell'_k:=\ell_k+d$ (where $d$ comes from~\eqref{eq:kappa}) and $Q_k:=Q_{\ell'_k}(\omega)$. By Lemma~\ref{lem:finite-covering}, we can choose at most $9m$ level-$\ell_k$ approximate squares whose images under $f$ cover $Q_k\cap f(K)$, and each of these images meets $Q_k\cap f(K)$. Passing to a subsequence, assume that the number of these squares is a constant $p\leq 9m$, say $E_{k,1},\ldots,E_{k,p}$. So 
    \begin{equation}\label{eq:cover-qkfk}
        Q_k \cap f(K) \subset \bigcup_{t=1}^p f(E_{k,t})
    \end{equation}
    and each $f(E_{k,t})$ meets $Q_k$. 
    For $1\leq t\leq p$, choose $x_{k,t}\in E_{k,t}$ with $f(x_{k,t})\in Q_k$. Choose a coding of $x_{k,t}$ within the symbolic approximate square defining $E_{k,t}$, and let $\eta_{k,t}\in[m]^{\N}$ be its row address. So $\pi_2(x_{k,t})=P_m(\eta_{k,t})$. See Figure~\ref{fig:last-case} for an illustration. 
    
    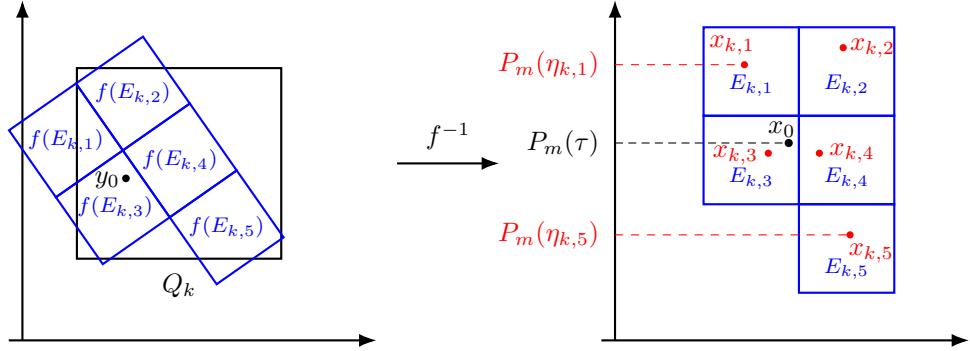
\begin{figure}[ht]
        \centering
        \begin{tikzpicture}[>=Latex,scale=0.9,every node/.style={font=\small}]
            \begin{scope}[xshift=0cm,yshift=0cm]
                \draw[thick,->] (0,0) -- (0,5.0);
                \draw[thick,->] (-0.2,0) -- (5.2,0);
                \draw[thick] (0.8,1.2) rectangle (3.8,4.0);
                \node at (2.3,0.8) {$Q_k$};
                \fill (1.52,2.38) circle (1.6pt);
                \node[left] at (1.58,2.36) {$y_0$};

                \begin{scope}[shift={(1.18,1.12)},rotate=35]
                    \draw[blue,thick] (0,0) rectangle (1.2,1.2);        
                    \draw[blue,thick] (1.2,0) rectangle (2.4,1.2);      
                    \draw[blue,thick] (0,1.2) rectangle (1.2,2.4);      
                    \draw[blue,thick] (1.2,1.2) rectangle (2.4,2.4);    
                    \draw[blue,thick] (1.2,-1.2) rectangle (2.4,0);     
                    \node[blue] at (0.6,1.8) {\scriptsize $f(E_{k,1})$};
                    \node[blue] at (1.8,1.8) {\scriptsize $f(E_{k,2})$};
                    \node[blue] at (0.63,0.58) {\scriptsize $f(E_{k,3})$};
                    \node[blue] at (1.8,0.6) {\scriptsize $f(E_{k,4})$};
                    \node[blue] at (1.8,-0.6) {\scriptsize $f(E_{k,5})$};
                \end{scope}
            \end{scope}

            \draw[->,thick] (5.5,2.6) -- (7,2.6);
            \node at (6.25,3.0) {$f^{-1}$};

            \begin{scope}[xshift=8.7cm,yshift=0cm]
                \draw[thick,->] (0,0) -- (0,5.0);
                \draw[thick,->] (-0.2,0) -- (5.2,0);

                \fill (2.55,2.9) circle (1.7pt);
                \node[right] at (2.1,3.1) {$x_0$};

                \draw[densely dashed] (0,2.9) -- (2.55,2.9);
                \node[left] at (-0.08,2.9) {$P_m(\tau)$};

                \draw[blue,thick] (1.3,3.3) rectangle (2.7,4.6);  
                \draw[blue,thick] (2.7,3.3) rectangle (4.1,4.6);  
                \draw[blue,thick] (1.3,2.0) rectangle (2.7,3.3);  
                \draw[blue,thick] (2.7,2.0) rectangle (4.1,3.3);  
                \draw[blue,thick] (2.7,0.7) rectangle (4.1,2.0);  

                \node[blue] at (2.0,3.75) {\scriptsize $E_{k,1}$};
                \node[blue] at (3.4,3.75) {\scriptsize $E_{k,2}$};
                \node[blue] at (2.0,2.38) {\scriptsize $E_{k,3}$};
                \node[blue] at (3.4,2.38) {\scriptsize $E_{k,4}$};
                \node[blue] at (3.4,1.05) {\scriptsize $E_{k,5}$};

                \fill[red] (1.9,4.05) circle (1.5pt);
                \node[red,above left] at (2.2,4.05) {$x_{k,1}$};
                \draw[red,dashed] (0,4.05) -- (1.9,4.05);
                \node[red,left] at (-0.08,4.05) {$P_m(\eta_{k,1})$};

                \fill[red] (3.35,4.3) circle (1.5pt);
                \node[red,above right] at (3.3,4.0) {$x_{k,2}$};

                \fill[red] (2.25,2.75) circle (1.5pt);
                \node[red,left] at (2.25,2.72) {$x_{k,3}$};

                \fill[red] (3,2.75) circle (1.5pt);
                \node[red,right] at (3,2.75) {$x_{k,4}$};

                \fill[red] (3.45,1.55) circle (1.5pt);
                \node[red,below right] at (3.28,1.58) {$x_{k,5}$};
                \draw[red,dashed] (0,1.55) -- (3.45,1.55);
                \node[red,left] at (-0.08,1.55) {$P_m(\eta_{k,5})$};
            \end{scope}
        \end{tikzpicture}
        \caption{A local cover and its pullback under $f^{-1}$ (color online)}
        \label{fig:last-case}
    \end{figure}

    Since both $f(x_{k,t})$ and $y_0=f(x_0)$ lie in $Q_k$, 
    \[
        |P_m(\eta_{k,t})-P_m(\tau)| = |\pi_2(x_{k,t})-\pi_2(x_0)| \leq |x_{k,t}-x_0| \leq r^{-1}\diam Q_k.
    \]
    Note that $\diam Q_k \leq \sqrt{2}n^{-\ell'_k}$ and $\ell'_k-\ell_k=d$. It then follows from~\eqref{eq:kappa} that 
    \[
        |P_m(\eta_{k,t})-P_m(\tau)| \leq r^{-1}\sqrt{2}n^{-\ell'_k} = \frac{\sqrt2}{\kappa}n^{-\ell_k} \leq n^{-\ell_k} < m^{-\widehat{\ell_k}+1} < m^{-\widehat{\ell_k}+3},
    \]
    that is, the requirements in Proposition~\ref{prop:good-scales} are met. Since there are only $p$ covering sets, we may pass to one further subsequence so that
    \begin{equation}\label{eq:source-row-tail-limits}
        \sigma^{\ell_k}\eta_{k,t}\to \eta_t, \quad 1\leq t\leq p.
    \end{equation}
    By Proposition~\ref{prop:good-scales}, every $\eta_t$ has infinitely many coordinates outside $J_*$. This conclusion is stable under further subsequences. Indeed, if $k_1<k_2<\cdots$ is a subsequence, we can extend (for each $t$) the family $\{\eta_{k_j,t}\}$ to all good scales, say $\{\widetilde{\eta}_{k,t}\}$, by setting $\widetilde{\eta}_{k_j,t}=\eta_{k_j,t}$ and $\widetilde{\eta}_{k,t}=\tau$ at the omitted indices $k$. The distance condition in Proposition~\ref{prop:good-scales} holds for $\widetilde\eta_{k,t}$: it was proved above at the retained indices, while at the omitted ones the distance is zero. Hence every limit point of $\sigma^{\ell_k}\widetilde\eta_{k,t}$ has infinitely many coordinates outside $J_*$. Since
    \[
        \sigma^{\ell_{k_j}}\widetilde\eta_{k_j,t} = \sigma^{\ell_{k_j}}\eta_{k_j,t}\to \eta_t,
    \]
    the same is true of $\eta_t$.

    We next identify the limiting sets. Let $\b_k$ be the row word fixed by $Q_k$ between levels $\ell'_k+1$ and $\widehat{\ell'_k}$, and let $\b_{k,t}$ be the block of $\eta_{k,t}$ at levels $\ell_k+1,\ldots,\widehat{\ell_k}$. As before, for $\ell\geq 1$ we write $\widetilde{D}_\ell= \diag(1,m^{s_\ell})$, where $s_\ell=\ell\log_m n-\widehat{\ell}\in (-1,0]$. By Lemma~\ref{lem:local-form}, there exist $z_k,z_{k,t}\in[-1,0]^2$ such that
    \[
        n^{\ell'_k}(Q_k-y_0) = \widetilde{D}_{\ell'_k}H(\b_k)+z_k \quad\text{and}\quad n^{\ell_k}(E_{k,t}-x_{k,t}) = \widetilde{D}_{\ell_k}H(\b_{k,t})+z_{k,t}.
    \]
    Applying $f$ to the second identity gives
    \[
        n^{\ell_k}(f(E_{k,t})-f(x_{k,t})) = rO\widetilde{D}_{\ell_k}H(\b_{k,t})+rOz_{k,t}.
    \]
    It follows that (recall that $\ell'_k=\ell_k+d$ and $\kappa=rn^d$)
    \begin{align*}
        n^{\ell'_k}(f(E_{k,t})-y_0) &= n^{\ell'_k}(f(E_{k,t})-f(x_{k,t})) + n^{\ell'_k}(f(x_{k,t})-y_0) \\
        &= \kappa O\widetilde{D}_{\ell_k}H(\b_{k,t})+a_{k,t},
    \end{align*}
    where $a_{k,t} := \kappa Oz_{k,t}+n^{\ell'_k}(f(x_{k,t})-y_0)$. Since both $f(x_{k,t})$ and $y_0$ belong to $Q_k$, $|a_{k,t}|\leq\sqrt{2}\kappa+\sqrt{2}$.

    The second diagonal entries of the matrices $\widetilde{D}_{\ell'_k}, \widetilde{D}_{\ell_k}$ lie in $(m^{-1},1]$, the translations $z_k,a_{k,t}$ are bounded, and the lengths of $\b_k$ tend to infinity. By compactness and a diagonal argument for the row words, we can extract a common subsequence on which
    \[
        \widetilde{D}_{\ell'_k}\to D_0,\, \widetilde{D}_{\ell_k}\to D_1,\, \b_k\to\xi,\, z_k\to z_0,\, \text{ and } a_{k,t}\to a_{0,t}.
    \]
    Both limiting matrices have the form $\operatorname{diag}(1,c)$ with $c\in[m^{-1},1]$. Also note that for each $1\leq t\leq p$, the word $\b_{k,t}$ is the prefix of $\sigma^{\ell_k}\eta_{k,t}$ of length $\widehat{\ell_k}-\ell_k\to\infty$. For any fixed number of digits, these words eventually agree with those in~\eqref{eq:source-row-tail-limits}. Thus $\b_{k,t}\to\eta_t$ coordinatewise.

    Recall that $\pi_2(K)=[0,1]$. By the above convergences and Lemmas~\ref{lem:fibreconverge} and~\ref{lem:hd-continuity},
    \begin{align*}
        n^{\ell'_k}(Q_k-y_0) & \to D_0(C(\xi)\times[0,1])+z_0 =: F, \\
        n^{\ell'_k}(f(E_{k,t})-y_0) &\to \kappa OD_1(C(\eta_t)\times[0,1])+a_{0,t} =: G_t.
    \end{align*}
    Let $T_k(z):=n^{\ell'_k}(z-y_0)$ and set
    \[
        \lambda_k := T_k\Pi\Big( \frac{\nu|_{C_{\ell'_k}(\omega)}}{\nu(C_{\ell'_k}(\omega))} \Big).
    \]
    Since $\spt\lambda_k=T_k(Q_k)\to F$, we may pass to a further subsequence such that $\lambda_k\to\lambda$ weakly. By Lemma~\ref{lem:fullsptlimit}, $\spt\lambda=F$. 
    The covering relation~\eqref{eq:cover-qkfk} and~\eqref{eq:density-final} then imply
    \[
        \lambda_k\Big( \bigcup_{t=1}^p T_k(f(E_{k,t})) \Big) \geq \frac{\nu(A\cap C_{\ell'_k}(\omega))}{\nu(C_{\ell'_k}(\omega))} \to 1.
    \]
    Since $T_k(f(E_{k,t}))\to G_t$ for every $t$, we have by Lemma~\ref{lem:finite-cover-limit} that
    \begin{equation*}
        F = \spt\lambda \subset \bigcup_{t=1}^p G_t.
    \end{equation*}
    Since $C(\xi)$ is non-empty and $D_0$ is diagonal and invertible, $F$ contains a nondegenerate vertical segment. By Baire's theorem, some $G_t$ contains a non-empty relatively open subset of $F$ and hence a nondegenerate vertical segment.
    
    On the other hand, writing $\eta_t:=j_1j_2\cdots$, we have 
    \[
        \mathcal{L}^1(C(\eta_t)) \leq \prod_{q=1}^N \frac{\#\Gamma_{j_q}}{n}, \quad \forall N\geq 1.
    \]
    Since there are infinitely many digits of $\eta_t$ lying outside $J_*$, the product tends to zero as $N\to\infty$. So $C(\eta_t)$ contains no nondegenerate interval. If $G_t$ contained a nondegenerate vertical segment, its inverse image under $z\mapsto\kappa OD_1z+a_{0,t}$ would be an oblique segment in $C(\eta_t)\times[0,1]$. This forces $C(\eta_t)$ to contain a segment, which leads to a contradiction.
\end{proof}

Propositions~\ref{prop:first-cases} and~\ref{prop:fullrow-case} establish Theorem~\ref{thm:nonoblique}. 

\begin{remark}\label{rem:full-no-interchange}
    When $J=[m]$, the arguments in Case~1 of the proof of Proposition~\ref{prop:first-cases} and in the preceding proof yield a stronger conclusion: for every self-affine measure $\mu$ with positive weights and every similitude $f(z)=rOz+a$,
    \[
        \mu(K\cap f(K))>0 \quad\Longrightarrow\quad O \text{ is diagonal}
    \]
    because these arguments also rule out anti-diagonal $O$. Indeed, in both cases\footnote{In particular, in the full-row case, the function $\pi_2f^{-1}$ still has a nonzero horizontal coefficient, so the choice of density point and scales remains valid.}, the limiting cover contains a nondegenerate vertical segment on the target side, while every covering set has the form $\kappa OD_1\bigl(C(\eta_t)\times[0,1]\bigr)+c_t$ with $\mathcal{L}^1(C(\eta_t))=0$.     For anti-diagonal $O$, the vertical projection of each such set is an affine image of $C(\eta_t)$. The covering relation would therefore place a nondegenerate interval inside a finite union of Lebesgue-null sets, a contradiction.
\end{remark}

\section{The conjecture of Elekes, Keleti and M\'ath\'e}
\label{sec:EKMconj}


In the sequel of the paper, we further adopt the following notation.
\begin{itemize}
    \item $I:=\{i\in[n]: \exists j\in[m] \text{ such that } (i,j)\in\Gamma\}$;
    \item $X:=\pi_1(K)=E(n,I)$ and $Y:=\pi_2(K)=E(m,J)$.
\end{itemize}

Theorem~\ref{thm:nonoblique} reduces the problem to the diagonal and anti-diagonal orthogonal parts. Both cases rely on a rigidity statement for one-dimensional deleted-digit sets, which we establish first. In the diagonal case, we compare two carpets with identical coordinate projections. In the anti-diagonal case, we exclude multiplicatively independent bases and show that the positive-measure intersection implies a product structure when the bases are dependent.

\subsection{Embedding rigidity for deleted-digit sets}

We use the notation $E(q,\D)$ from~\eqref{eq:deleted-digit}. The following lemma converts an inclusion modulo integer translations into equality up to translation. Its proof adapts the argument of Algom and Hochman~\cite[Proposition~3.1]{AH19} to two digit sets of equal cardinality and to an inclusion modulo $\Z$.

\begin{lemma}\label{lem:deleted-digit-translation}
    Let $q\geq2$ and let $\D,\D'\subset[q]$ be non-empty digit sets with $\#\D=\#\D'$. If
    \[
        E(q,\D')\subset E(q,\D)+c+\Z
    \]
    for some $c\in\R$, then $E(q,\D')$ is a translate of $E(q,\D)$.
\end{lemma}
\begin{proof}
    Since $E(q,\D-\min \D)$ is a translate of $E(q,\D)$, we may assume that $\min\D=\min\D'=0$. It suffices to show that $\D=\D'$. For convenience, write $A:=E(q,\D)$ and $B:=E(q,\D')$. The cases $\#\D=1$ and $\#\D=q$ are immediate, so suppose that $2\leq\#\D<q$. 

    Let $t:=\{-c\}\in[0,1)$ be the fractional part of $-c$. So $B+t\subset A+\Z$. Since $0\in A\cap B$, we have $t\in A$ and
    \begin{equation}\label{eq:mod-translation}
        \{x+t\}\in A, \quad \forall x\in B.
    \end{equation}
    The key is noting that $t$ must be a $q$-adic rational. Otherwise, write its unique base-$q$ expansion as $t=\sum_{k\geq1}i_kq^{-k}$. Then $i_k\in\D$ for every $k$. For $d'\in\D'$, the point $d'q^{-k}$ belongs to $B$, and $\{t+d'q^{-k}\}$ is again non-$q$-adic. Its $k$-th digit is $i_k+d'$ modulo $q$. By~\eqref{eq:mod-translation}, $i_k+\D'\pmod q\subset\D$. Since $\#\D=\#\D'$,
    \begin{equation}\label{eq:digit-translate}
        \D=i_k+\D'\pmod q, \quad k\geq 1.
    \end{equation}

    We claim that $i_k+d'<q$ for $k\geq2$ and $d'\in\D'$. Assume on the contrary that some $i_k+d'\geq q$. Clearly, $i_k+d'\leq 2q-2$. So $i_k+d'=q+r_k$ for some $0\leq r_k\leq q-2$. For any $e\in\D'$, 
    \begin{align*}
        t+\Big( \frac{e}{q^{k-1}}+\frac{d'}{q^k} \Big) &= \frac{i_1}{q}+\cdots+\frac{i_{k-2}}{q^{k-2}}+\frac{i_{k-1}+e}{q^{k-1}}+\frac{i_k+d'}{q^k} + \frac{i_{k+1}}{q^{k+1}} + \cdots \\
        &= \frac{i_1}{q}+\cdots+\frac{i_{k-2}}{q^{k-2}}+\frac{i_{k-1}+e+1}{q^{k-1}}+\frac{r_k}{q^k} + \frac{i_{k+1}}{q^{k+1}} + \cdots.
    \end{align*}
    So the $(k-1)$-st digit of this number is $i_{k-1}+e+1\pmod q$. By~\eqref{eq:mod-translation} (taking $x$ as $\frac{e}{q^{k-1}}+\frac{d'}{q^k}\in B$) and the uniqueness of such an expansion, $i_{k-1}+e+1\pmod q\in\D$. Since $e\in\D'$ was arbitrary,
    \[
        i_{k-1}+\D'+1\pmod q \subset \D.
    \]
    Comparing cardinalities again, $\D=i_{k-1}+\D'+1\pmod q$. By~\eqref{eq:digit-translate}, we obtain 
    \[
        \D'+1=\D'\pmod q.
    \]
    Then $\D'=[q]$, contrary to $\#\D'<q$. This proves the claim.

    Now~\eqref{eq:digit-translate} becomes $\D=i_k+\D'$ for every $k\geq2$. Since $\min\D=\min\D'=0$, $i_k=0$ for every $k\geq2$. Then $t$ is $q$-adic, which leads to a contradiction.

    Choose $N\geq0$ such that $q^Nt\in\Z$. Since $0\in\D'$, we have $q^{-N}B\subset B$. So $q^{-N}B+t \subset A+\Z$. Multiplying by $q^N$ and using $q^NA\subset A+\Z$, we obtain
    \[
        B\subset q^NA+\Z\subset A+\Z.
    \]
    Hence $B\cap(0,1)\subset A$. Since $B$ is perfect and $A$ is closed, $B\subset A$. Comparing the first-level cylinders, we have $\D'\subset\D$ and equality follows from $\#\D'=\#\D$.
\end{proof}

\begin{proposition}\label{prop:deleted-rigidity}
    Let $q\geq2$ and let $\D,\D'\subset[q]$ be non-empty with $\#\D=\#\D'$. Let $\lambda$ be a self-similar measure on $E(q,\D)$ with positive weights. If $\psi(x)=\varepsilon x+c$, where $\varepsilon\in\{\pm 1\}$, $c\in\R$, and
    \[
        \lambda\bigl(E(q,\D)\cap\psi(E(q,\D'))\bigr)>0,
    \]
    then $E(q,\D)=\varepsilon E(q,\D')+c'$ for some $c'\in\R$.
\end{proposition}
\begin{proof}
    The cases $\#\D=1$ and $\#\D=q$ are immediate. Since
    \[
        1-E(q,\D')=E(q,q-1-\D'),
    \]
    reflecting $\D'$ reduces the remaining cases to $\varepsilon=1$. Write $A:=E(q,\D)$ and $B:=E(q,\D')$. So $\lambda(B+c)>0$.

    Since $\#\D>1$, $\lambda$ is non-atomic. In particular, it gives zero mass to $q$-adic rationals. By the one-dimensional version of Lemma~\ref{lem:density}, we can find a density point $x\in A\cap(B+c)$ with a unique base-$q$ expansion $x=\sum_{k\geq 1} i_kq^{-k}$ with $i_k\in\D$. For $N\geq 1$, let $a_N := \sum_{k=1}^N i_kq^{-k}$ and $A_N := q^{-N}A+a_N$. Then $A_N$ is the unique level-$N$ cylinder of $A$ containing $x$.
    So
    \[
        \lim_{N\to\infty} \frac{\lambda(A_N\cap(B+c))}{\lambda(A_N)} = 1.
    \]
    Define $T_N(y):=q^N(y-a_N)$, $y\in\R$. It is easy to check that $\lambda=T_N(\frac{\lambda|_{A_N}}{\lambda(A_N)})$, which gives  
    \[
        \lambda(T_N(B+c)) = \frac{\lambda(A_N\cap (B+c))}{\lambda(A_N)}.
    \]
    Writing $t_N:=q^N(c-a_N)$, we have $T_N(B+c)=q^NB+t_N$ and thus $\lambda(q^NB+t_N) \to 1$ as $N\to\infty$. 
    Since $q^NB\subset B+\Z$, it follows that 
    \begin{equation}\label{eq:periodic-cover}
        \lambda(B+t_N+\Z) \to 1, \quad N\to\infty.
    \end{equation}

    Passing to a subsequence, we may assume that $t_N\to t$ modulo $\Z$. Then $A\subset B+t+\Z$. This is because if a point of $A$ lies outside this closed set, a sufficiently small neighborhood of that point would miss $B+t_N+\Z$ for all large $N$. But that neighborhood has positive $\lambda$-measure, contradicting~\eqref{eq:periodic-cover}. Applying Lemma~\ref{lem:deleted-digit-translation} (with the two digit sets interchanged), we establish the proposition.
\end{proof}

For the case of different bases, we need only an embedding as described below.

\begin{lemma}\label{lem:deleted-embedding}
    Let $A=E(q_1,\D_1)$ and $B=E(q_2,\D_2)$, where $q_1,q_2\geq2$ and both digit sets contain at least two elements. Let $\lambda$ be a self-similar measure on $A$ with positive weights. If $\psi$ is an invertible similitude and $\lambda(A\cap\psi(B))>0$, then there is an invertible similitude $h$ such that $h(A)\subset B$.
\end{lemma}
\begin{proof}
    As in the previous proof, the cylinder density lemma provides a density point $x\in A\cap\psi(B)$ with a unique base-$q_1$ expansion $x=\sum_{k\geq 1} i_kq_1^{-k}$, where $i_k\in\D_1$. For $N\geq 1$, let $a_N := \sum_{k=1}^N i_kq_1^{-k}$ and $A_N := q_1^{-N}A+a_N$. 
    Then
    \[
        \lim_{N\to\infty} \frac{\lambda(A_N\cap\psi(B))}{\lambda(A_N)} = 1.
    \]
    The map $T_N(y):=q_1^N(y-a_N)$ sends the normalized restriction of $\lambda$ to $A_N$ to $\lambda$. So $\lambda(T_N\psi(B))\to 1$. Writing $\psi(z)=\alpha z+c$ and $t_N:=q_1^N(c-a_N)$, we have $T_N\psi(B)=q_1^N\alpha B+t_N$ and hence $\lambda(q_1^N\alpha B+t_N) \to 1$. For all large $N$, pick an integer $k_N\geq 0$ such that $1\leq q_1^Nq_2^{-k_N}<q_2$. Since $q_2^{k_N}B\subset B+\Z$, we have 
    \[
        q_1^N\alpha B+t_N \subset q_1^{N}\alpha q_2^{-k_N}\Big( B+\frac{t_N}{q_1^{N}\alpha q_2^{-k_N}}+\Z \Big) =: E_N.
    \]
    Then $\lambda(E_N)\to 1$. 
    
    Passing to a subsequence, we may assume that $q_1^Nq_2^{-k_N}\to \beta$ and $\frac{t_N}{q_1^{N}\alpha q_2^{-k_N}}\pmod\Z\to t$. Let $E:=\alpha \beta(B+t+\Z)$ be the ``limiting set''. Then it is not hard to see that every compact set disjoint from $E$ is disjoint from $E_N$ for all large $N$. As before, together with $\lambda(E_N)\to 1$ and $\spt\lambda=A$, this implies that $A\subset \alpha \beta(B+t+\Z)$. By Baire's theorem, some $\alpha \beta(B+t+u)$ contains a non-empty relatively open subset of $A$ and hence a cylinder copy $\varphi(A)$. Then 
    \[
        \alpha ^{-1}\beta^{-1}\varphi(A)-t-u\subset B,
    \]
    as desired.
\end{proof}

\begin{remark}\label{rem:ratio-log}
    The similarity ratio of $h$ constructed above equals $|\alpha ^{-1}\beta^{-1}q_1^{-k}|$ for some integer $k\geq 1$, where $\beta$ is an accumulation value of the sequence $\{q_1^{N}q_2^{-k_N}\}_{N\geq 1}$. By our choice of $k_N$, if $q_1=q_2$ then $k_N=N$ and thus $\beta=1$; so the similarity ratio of $h$ belongs to $|\alpha |^{-1}q_1^{\Z}$. If $\frac{\log q_1}{\log q_2}\in\Q$, then $q_2$ is simply a rational power of $q_1$ and hence $\{q_1^Nq_2^{-k_N}\}_{N\geq 1}$ is a finite set. Consequently, any subsequential limit of this sequence belongs to $q_1^{\Q}$. Then the similarity ratio of $h$ belongs to $|\alpha |^{-1}q_1^{\Q}$. We shall use these observations in the final section.
\end{remark}

Before moving on, let us address the degenerate case of Theorem~\ref{thm:ekmconj}. If $K$ is contained in a line $L$, then the direction of $L$ is preserved by $\diag(n^{-1},m^{-1})$. Since $n\neq m$, $L$ must be either horizontal or vertical, and $K$ is a copy of a deleted-digit set on $L$ with positive Hausdorff dimension. Since $\mu$ is non-atomic, a positive-measure intersection implies that $f(L)=L$. The restriction of $f$ to $L$ is a one-dimensional isometry, and Proposition~\ref{prop:deleted-rigidity} shows that $f(K)$ is a translate of $K$. 

For the remainder of the section, we always assume that $K$ is not contained in a line. Thus $\#I,\#J\geq2$, and Lemma~\ref{lem:line-null} implies that every line has $\mu$-measure zero. In particular, geometric cylinders only overlap on sets of measure zero, and their normalized restrictions are the affine images of $\mu$ corresponding to those cylinders.

\subsection{Axis-preserving isometries}

The following comparison result adapts the argument of~\cite[Proposition 7.1]{EKM10} to two patterns.

\begin{proposition}\label{prop:rigidity-carpet}
    Let $\Gamma'\subset[n]\times[m]$ satisfy $\#\Gamma'=\#\Gamma$ and write $K'=K(n,m,\Gamma')$. Assume that $\pi_i(K')=\pi_i(K)$, $i=1,2$. If $\mu(K\cap (K'+t))>0$ for some $t\in\R^2$, then $K'=K$.
\end{proposition}
\begin{proof}
	If $\mu(K\cap (K'+u))=1$ for some $u=(u_1,u_2)\in\R^2$, then $K\subset K'+u$. It follows that 
	\[
	X = \pi_1(K) \subset \pi_1(K'+u) = X+u_1
	\]
	and (similarly) $Y\subset Y+u_2$. Thus $u_1=u_2=0$; then $K\subset K'$. Since $\#\Gamma'=\#\Gamma$, we must have $\Gamma'=\Gamma$, which gives $K'=K$.
	
   Assume that $\mu(K\cap (K'+u))<1$ for all $u\in\R^2$. Note that the map $g(u)=\mu(K\cap(K'+u))$ is upper semicontinuous on $\mathbb{R}^2$ and vanishes outside the compact difference set $K-K'$. So $g$ attains the maximum on $\mathbb{R}^2$ and there is a constant $\delta>0$ such that 
	\begin{equation}\label{eq:massdrop}
		\sup_{u\in\R^2} \mu(K\cap (K'+u)) \leq 1-\delta.
	\end{equation}
	
	Write $M:=\diag(n,m)$. Let $\Gamma_2 := M\Gamma+\Gamma$ and $\Gamma'_2:=M\Gamma'+\Gamma'$. Then $K=K(n^2,m^2,\Gamma_2)$ and $K'=K(n^2,m^2,\Gamma'_2)$. Choose $0<\varepsilon<\delta/8$ and let $p_*:=\min_{(i,j)\in\Gamma} p_{i,j}$. By Lemma~\ref{lem:density}, there is a cylinder $Q$ (with respect to the system associated with the triple $(n^2,m^2,\Gamma_2)$) such that $\mu(Q \cap (K'+t)) > \big( 1-\varepsilon p_*^2 \big)\mu(Q)$. Thus $\mu(Q\setminus (K'+t))< \varepsilon p_*^2 \mu(Q)$. Denote by $\phi_Q$ the cylinder map associated with $Q$. For $e\in\Gamma_2$, write $\phi_e(z):=M^{-2}(z+e)$. From $\phi_Q\phi_e(K)\subset Q$, we see that
    \[
		\mu\big(\phi_Q\phi_e(K)\setminus (K'+t)\big) <\varepsilon p_*^2\mu(Q)\leq \varepsilon\mu(\phi_Q\phi_e(K)).
	\]
	Thus 
	\[
		\mu\big( \phi_Q\phi_{e}(K)\cap (K'+t)  \big) > (1-\varepsilon)\mu(\phi_Q\phi_{e}(K)).
	\]

    Applying $(\phi_Q\phi_e)^{-1}$ to both sides of the above inequality, we get
    \begin{equation}\label{eq:density-rescaling}
        \mu\big( K\cap (\phi_Q\phi_e)^{-1}(K'+t) \big) > 1-\varepsilon, \quad \forall e\in\Gamma_2.
    \end{equation}
    Write $\phi_Q(z)=M^{-2k}z+s$. Then $(\phi_Q\phi_e)^{-1}(z)=M^{2(k+1)}(z-s)-e$. Since $K'=K(n^2,m^2,\Gamma'_2)$, it is not hard to see that $M^{2(k+1)}K'\subset K'+\Gamma'_2+M^2\Z^2$. Writing $u:=M^{2(k+1)}(t-s)$, we then have
    \begin{align*}
        (\phi_Q\phi_e)^{-1}(K'+t) &= M^{2(k+1)}K'+u-e \\
        &\subset K'+ u-e+\Gamma'_2+M^2\Z^2 \\
        &= \bigcup_{v\in\Z^2: e+v\in\Gamma'_2+M^2\Z^2} (K'+u+v).
    \end{align*}
    So~\eqref{eq:density-rescaling} implies
	\begin{equation}\label{eq:density-children2}
		\sum_{v\in\Z^2: e+v\in\Gamma'_2+M^2\Z^2} \mu(K\cap(K'+u+v)) > 1-\varepsilon, \quad \forall e\in\Gamma_2.
	\end{equation}
	Here distinct integer translation copies of $K'$ may meet only on the coordinate boundary lines, which have $\mu$-measure zero (recall Lemma~\ref{lem:line-null}). 
	
	Since $K,K'\subset[0,1]^2$, there are at most four integer vectors $v$ with $\mu(K\cap (K'+u+v))>0$. Enumerate them as $v_{1},\ldots,v_{p}$ (where $p\leq 4$) so that $\beta_j := \mu(K\cap (K'+u+v_j))>0$ is decreasing. Inequality~\eqref{eq:density-children2} implies $1-\varepsilon\leq \sum_{j=1}^{p} \beta_j \leq 1$. On the other hand, we have by~\eqref{eq:massdrop} that $\beta_{1}\leq 1-\delta$. By our previous choice of $\varepsilon$, $p\geq 2$ and 
	\[
	    \beta_1 \geq \beta_2 \geq \frac{\sum_{j=1}^p \beta_j-\beta_1}{p-1} > \frac{\delta-\varepsilon}{3} > \varepsilon.
	\]
	So for every $e\in\Gamma_2$, $v_{1}$ and $v_{2}$ always appear in the summation in~\eqref{eq:density-children2}. Then 
	\[
	    e+v_{k} \in \Gamma'_2+M^2\Z^2, \quad \forall k=1,2, \forall e\in\Gamma_2.
	\]
	So $\Gamma_2+v_k\subset \Gamma'_2+M^2\Z^2$. Both sides consist of $(\#\Gamma)^2$ distinct residue classes modulo $M^2\Z^2$ and thus $\Gamma_2+v_k+M^2\mathbb{Z}^2=\Gamma_2'+M^2\mathbb{Z}^2$, $k=1,2$. So    
	\[
	    \Gamma_2 + (v_2-v_1) +M^2\Z^2 = \Gamma_2+M^2\Z^2.
	\]
	
	Clearly, $v_2-v_1\in\{0,\pm 1\}^2\setminus\{(0,0)\}$. By~\cite[Lemma 7.3 and Remark 7.6]{EKM10} and because of $n\neq m$, either $K=E(n,I)\times[0,1]$ or $K=[0,1]\times E(m,J)$. Combining this with $\#\Gamma'=\#\Gamma$ and $\pi_i(K')=\pi_i(K)$, $i=1,2$, we have $\Gamma'=\Gamma$ and thus $K'=K$.
\end{proof}

\begin{proposition}
    Assume that $O=\diag(o_1, o_2)$,  where $o_1,o_2\in\{\pm 1\}$ and $O$ is not the identity matrix. If $f(z)=Oz+a$ satisfies $\mu(K\cap f(K))>0$, then $K$ and $OK$ differ by a translation.
\end{proposition}
\begin{proof}
    Write $a=(a_1,a_2)$. From $\mu(K\cap f(K))>0$ we see that 
    \[
        \pi_1\mu(X\cap (o_1X+a_1))>0 \quad\text{and}\quad \pi_2\mu(Y\cap(o_2Y+a_2))>0.
    \]
    Both $\pi_1\mu$ and $\pi_2\mu$ are self-similar measures with positive weights. By Proposition~\ref{prop:deleted-rigidity}, $X$ and $o_1X$ (resp. $Y$ and $o_2Y$) differ by a translation.

    To make this precise, note that the left endpoints of $X$ and $-X$ are $\frac{\min I}{n-1}$ and $-\frac{\max I}{n-1}$, respectively. So $o_1X+b_1=X$ and $o_2Y+b_2=Y$, where 
    \[
        b_1:=\frac{1-o_1}{2}\cdot\frac{\min I+\max I}{n-1} \quad\text{and}\quad b_2:=\frac{1-o_2}{2}\cdot\frac{\min J+\max J}{m-1}.
    \]
    Let $b:=(b_1,b_2)$. 
    For any $(i,j)\in\Gamma$, set $i'=o_1i+(n-1)b_1$ and $j'=o_2j+(m-1)b_2$. Then $i'\in[n]$, $j'\in[m]$, and 
    \[
        o_1\cdot\frac{x+i}{n}+b_1 = \frac{(o_1x+b_1)+i'}{n} \quad\text{and}\quad o_2\cdot\frac{y+j}{m}+b_2 = \frac{(o_2y+b_2)+j'}{m}
    \]
    for all $(x,y)\in\R^2$. Consequently, $O\phi_{i,j}(K)+b=\phi_{i',j'}(OK+b)$.
    
    It follows directly from the above calculation that $OK+b$ is the attractor of $\{\phi_{i',j'}\}_{(i',j')\in\Gamma'}$, where
    \[
        \Gamma' := \Big\{ (o_1i+(n-1)b_1, o_2j+(m-1)b_2): (i,j)\in\Gamma \Big\} \subset [n]\times[m].
    \]
    In other words, $OK+b$ equals the carpet $K'=K(n,m,\Gamma')$. So $\pi_1(K')=o_1X+b_1=X$, $\pi_2(K')=o_2Y+b_2=Y$, $\#\Gamma'=\#\Gamma$, and $f(K) = OK+a=K'+a-b$. By Proposition~\ref{prop:rigidity-carpet}, $K'=K$ and hence $OK=K-b$.
\end{proof}

This settles the diagonal case of Theorem~\ref{thm:ekmconj}.

\subsection{Axis-interchanging isometries}
\label{subsec:axis-interc-isom}

The multiplicatively independent case follows from the one-dimensional embedding obstruction. Recall that we always assume that $K$ is not contained in a line and $\#I,\#J\geq2$.

\begin{proposition}\label{prop:independent-no-interchange}
    For every axis-interchanging isometry $f$, if $\mu(K\cap f(K))>0$ then $\frac{\log n}{\log m}\in\Q$.
\end{proposition}
\begin{proof}
    Since $\mu(K\cap f(K))>0$, we know from Remark~\ref{rem:full-no-interchange} that $J\neq[m]$ and hence $\dimh Y<1$. Since $f$ is axis-interchanging, we may write
    \[
        f^{-1}(x,y)=(\delta_1y+c_1,\delta_2x+c_2),
    \]
    where $\delta_1,\delta_2\in\{\pm 1\}$. For every $(x,y)\in K\cap f(K)$, we have $\delta_2x+c_2\in Y$. Equivalently, $x\in\delta_2(Y-c_2)$. So $\pi_1(K\cap f(K))\subset X\cap (\delta_2Y-\delta_2c_2)$ and thus
    \[
        \pi_1\mu\bigl(X\cap(\delta_2Y-\delta_2c_2)\bigr)>0.
    \]
    By Lemma~\ref{lem:deleted-embedding}, there exists an invertible affine map $h$ such that $h(X)\subset Y$. Since $X$ is non-singleton, whereas $Y$ is homogeneous, satisfies the open set condition and has dimension less than one, \cite[Corollary 7.6]{Shm19} implies that $\frac{\log n}{\log m}\in\Q$.
\end{proof}

To complete the proof, we need a simple compactness observation.

\begin{lemma}\label{lem:compact-isometry}
    Let $A\subset\T$ be a non-empty compact set and let $h$ be an isometry on $\T$. If $h(A)\subset A$ then $h(A)=A$.
\end{lemma}
\begin{proof}
    This is easy; see~\cite[Exercise 28.6]{Mun00} for example.
\end{proof}

\begin{proposition}\label{prop:anti-product}
    Suppose $\frac{\log n}{\log m}\in\Q$ and that $f$ is an axis-interchanging isometry. For any self-affine measure $\mu$ with positive weights, if $\mu(K\cap f(K))>0$ then $\Gamma=I\times J$, $K=X\times Y$ and $\dimh X=\dimh Y$.
\end{proposition}

In what follows, we use a subscript $\T$ to denote reduction modulo $\Z^2$ or $\Z$ (according to the context). 

\begin{proof}
    Choose integers $1\leq p<q$ such that $n^p=m^q$. As before, write
    \[
        f^{-1}(x,y) = (\delta_1y+c_1,\delta_2x+c_2),
    \]
    where $\delta_1,\delta_2\in\{-1,1\}$ and $c_1,c_2\in\R$. Let $M: (x,y) \mapsto (nx, my) \pmod{\Z^2}$. Clearly, $M(K_{\T})=K_{\T}$.

    Write $A:=\Pi^{-1}(K\cap f(K))$. By Lemma~\ref{lem:density}, we can choose a cylinder density point $\omega\in A$. For $t\geq1$, set $w_t:=\omega|_{p^2t}$ and $A_t:=\{\eta\in\Sigma: w_t\eta\in A\}$. So
    \begin{equation}\label{eq:anti-cylinder-density}
        \nu(A_t) = \frac{\nu(A\cap[w_t])}{\nu([w_t])} \to 1, \quad t\to\infty.
    \end{equation}
    Consider $G_t:=M^{pqt}f^{-1}\phi_{w_t}$. It is easy to compute that
    \[
        G_t(x,y)= (\delta_1m^{(q^2-p^2)t}y+\alpha_t,\delta_2x+\beta_t) \pmod{\Z^2}
    \]
    for some $\alpha_t,\beta_t\in\T$. The original horizontal coordinate now appears in the second coordinate at scale one, whereas the vertical one is multiplied by $m^{(q^2-p^2)t}$ in the first. Note that (recall that $\Pi(\sigma\eta)=M\Pi(\eta)\pmod{\Z^2}$)
    \begin{equation}\label{eq:symbolic-Gt}
        G_t(\Pi(\eta)) = \big(\delta_1\pi_2\big(\Pi(\sigma^{(q^2-p^2)t}\eta)\big)+\alpha_t, \delta_2\pi_1(\Pi\eta)+\beta_t\big) \pmod{\Z^2}, \quad \forall \eta\in\Sigma.
    \end{equation}
    Passing to a subsequence, we may assume that $\alpha_t\to\alpha$ and $\beta_t\to\beta$. Let
    \[
        g(y) := \delta_1y+\alpha \quad\text{and}\quad h(x) := \delta_2x+\beta, \quad x,y\in\T.
    \]
    We claim that
    \begin{equation}\label{eq:anti-circle-product}
        g(Y_{\T})\times h(X_{\T})\subset K_{\T}.
    \end{equation}

    The proof uses the same disjoint-block observation as in Lemma~\ref{lem:fullsptlimit}. Suppose that~\eqref{eq:anti-circle-product} fails. Then there are $x_0\in X_{\T}$, $y_0\in Y_{\T}$ and $\varepsilon>0$ such that the open set
    \[
        U := \tilde{B}(g(y_0),\varepsilon) \times \tilde{B}(h(x_0),\varepsilon)
    \]
    is disjoint from $K_{\T}$. Here $\tilde{B}$ denotes a ball for the usual distance on $\T$, that is, $\dist_{\T}(a,b):=\min_{z\in\Z}|a-b-z|$. Fix a positive integer $L$ so large that $m^{-L},n^{-L}<\varepsilon/2$. Let $i_1\cdots i_L$ and $j_1\cdots j_L$ be the length-$L$ prefixes of addresses of $x_0$ and $y_0$, respectively. For all large $t$ with $(q^2-p^2)t\geq L$, let $E_t\subset\Sigma$ consist of sequences whose first $L$ horizontal digits are $i_1,\ldots,i_L$ and whose vertical digits at levels $(q^2-p^2)t+1,\ldots,(q^2-p^2)t+L$ are $j_1,\ldots,j_L$. Note that
    \[
        \nu(E_t) = \prod_{k=1}^L \Big( \sum_{j:(i_k,j)\in\Gamma}p_{i_k,j} \Big) \prod_{k=1}^L q_{j_k} =: c_L > 0,
    \]
    independently of $t$. For every $\eta\in E_t$, 
    \[
         \dist_{\T}(\pi_1(\Pi(\eta)), x_0) \leq n^{-L}<\varepsilon/2 \quad\text{and}\quad
         \dist_{\T}\big(\pi_2(\Pi(\sigma^{(q^2-p^2)t}\eta)),y_0\big) \leq m^{-L}<\varepsilon/2.
    \]
    Since $\alpha_t\to\alpha$ and $\beta_t\to\beta$, for all large $t$ we have $\dist_{\T}(\alpha_t,\alpha)<\varepsilon/2$ and $\dist_{\T}(\beta_t,\beta)<\varepsilon/2$. Together with~\eqref{eq:symbolic-Gt}, we see that $G_t(\Pi(E_t))\subset U$. 

    On the other hand, 
    \[
        G_t(\Pi(A_t)) = M^{pqt}f^{-1}\phi_{w_t}(\Pi(A_t)) \subset M^{pqt}f^{-1}(\Pi(A)) \subset M^{pqt}K_{\T} \subset K_{\T}.
    \] 
    Since $U\cap K_{\T}=\varnothing$, we have for all large $t$ that $E_t\cap A_t=\varnothing$ and thus
    \[
        \nu(A_t)\leq1-\nu(E_t)=1-c_L.
    \]
    This contradicts~\eqref{eq:anti-cylinder-density} and proves~\eqref{eq:anti-circle-product}.

    Projecting the inclusion~\eqref{eq:anti-circle-product}, we get $g(Y_{\T})\subset X_{\T}$ and $h(X_{\T})\subset Y_{\T}$. So
    \[
        gh(X_{\T}) \subset g(Y_{\T})\subset X_{\T}.
    \]
    Lemma~\ref{lem:compact-isometry} makes both inclusions equalities. Thus $g(Y_{\T})=X_{\T}$, $h(X_{\T})=Y_{\T}$, and~\eqref{eq:anti-circle-product} gives $X_{\T}\times Y_{\T}\subset K_{\T}$. The opposite inclusion always holds and hence $K_{\T}=X_{\T}\times Y_{\T}$. Since the quotient map is injective on $(0,1)^2$, 
    \[
        (X\cap(0,1))\times(Y\cap(0,1))\subset K.
    \]
    Both $X$ and $Y$ are non-singleton deleted-digit sets. In particular, they are perfect, and $X\cap(0,1)$ and $Y\cap(0,1)$ are dense in $X$ and $Y$. Taking closures, we get $K=X\times Y$. So $\Gamma=I\times J$. Finally, $g(Y_{\T})=X_{\T}$ implies $\dimh X=\dimh Y$.
\end{proof}

The product structure reduces the remaining rigidity statement to the one-dimensional result at a common base.

\begin{proof}[Proof of Theorem~\ref{thm:ekmconj}]
    The degenerate case was proved earlier. So we may assume that $K$ is not contained in any line. By Theorem~\ref{thm:nonoblique}, $O$ is either diagonal or anti-diagonal. The diagonal case was established in the previous subsection.

    For the anti-diagonal case, again write 
    \[
        f(x,y) = (\delta_1y+a_1, \delta_2 x+a_2),
    \]
    where $\delta_1,\delta_2\in\{\pm 1\}$. Since $\mu(K\cap f(K))>0$, we have by Proposition~\ref{prop:independent-no-interchange} that $\frac{\log n}{\log m}\in\Q$ and by Proposition~\ref{prop:anti-product} that $K=X\times Y$ and $\dimh X=\dimh Y$. Choose positive integers $p,q$ such that $n^p=m^q$. Then $X=E(n^p,I_p)$ and $Y=E(m^q,J_q)$, where $I_p:=\{\sum_{k=1}^p i_kn^{p-k}: i_k\in I\}$ and $J_q:=\{\sum_{k=1}^q j_km^{q-k}: j_k\in J\}$, respectively. Note that $\#I_p=\#J_q$ and 
    \[
        K \cap f(K) = (X\cap(\delta_1Y+a_1))\times (Y\cap(\delta_2 X+a_2)).
    \]
    Since it has positive $\mu$-measure, 
    \[
        \pi_1\mu(X\cap(\delta_1Y+a_1))>0 \quad\text{and}\quad \pi_2\mu(Y\cap(\delta_2 X+a_2))>0.
    \]
    These are self-similar measures with positive weights for the corresponding systems. By Proposition~\ref{prop:deleted-rigidity}, there are $c_1,c_2\in\R$ such that $X=\delta_1 Y+c_1$ and $Y=\delta_2X+c_2$. So 
    \[
        f(K) = (\delta_1Y+a_1) \times (\delta_2X+a_2) = (X+a_1-c_1) \times (Y+a_2-c_2),
    \]
    which is a translated copy of $K$.
\end{proof}

\section{Logarithmic commensurability}
\label{sec:log}

In this section, we fix $K=K(n,m,\Gamma)$, $\mu=\mu_{\mathbf{p}}$ and $f(z)=rOz+a$ (where $r>0$) and assume that $\mu(K\cap f(K))>0$ as in Theorem~\ref{thm:log}. We also retain the notation $X=\pi_1(K)$ and $Y=\pi_2(K)$. We shall use the following special case of the logarithmic commensurability theorem of Feng and Wang~\cite[Theorem 1.1]{FW09}.

\begin{lemma}[\cite{FW09}]\label{lem:FW09}
    Let $E=E(q,\D)$ be a deleted-digit set with $1<\#\D<q$. If $h$ is a similitude sending $E$ into itself, then the similarity ratio of $h$ is a rational power of $q$.
\end{lemma}

By Theorem~\ref{thm:nonoblique}, $O$ is either diagonal or anti-diagonal. The latter case is easy to deal with.

\begin{proposition}\label{prop:log-1prime}
	If $O$ is anti-diagonal then $\frac{\log n}{\log m}\in\Q$ and $r\in n^{\Q}=m^{\Q}$.
\end{proposition}
\begin{proof}
    If either $\#I=1$ or $\#J=1$, then $K$ is contained in a horizontal or vertical line, while $f(K)$ is contained in a perpendicular line. Their intersection therefore contains at most one point and has zero $\mu$-measure, a contradiction. Hence $\#I,\#J\geq2$.

	Since $O$ is anti-diagonal, there are $\varepsilon_1,\varepsilon_2\in\{-1,1\}$ such that 
	\[
	    f(x,y) = (r\varepsilon_1y+a_1, r\varepsilon_2x+a_2), \quad (x,y)\in\R^2.
	\]
	From $\mu(K\cap f(K))>0$ we infer that 
	\[
	    \pi_1\mu(X\cap (r\varepsilon_1 Y+a_1))>0 \quad\text{and}\quad \pi_2\mu(Y\cap(r\varepsilon_2 X+a_2))>0.
	\]
	Lemma~\ref{lem:deleted-embedding} yields two similitudes $h_1,h_2$ such that $h_1(Y)\subset X$ and $h_2(X)\subset Y$. From Remark~\ref{rem:full-no-interchange}, we see that $J\neq [m]$. Thus both $X$ and $Y$ must be proper deleted-digit sets. Applying Shmerkin's theorem~\cite[Corollary 7.6]{Shm19}, we have $\frac{\log n}{\log m}\in\Q$. Note that $h_1h_2(X)\subset X$. By Remark~\ref{rem:ratio-log}, the similarity ratio of $h_1h_2$ belongs to $r^{-2}n^{\Q}$. Then Lemma~\ref{lem:FW09} yields $r\in n^{\Q}=m^{\Q}$.
\end{proof}

For the diagonal case, we first consider when at least one of $X$ and $Y$ is a proper deleted-digit set, which is equivalent to that $\min\{\dimh X, \dimh Y\}<1$. The proof is similar as above.

\begin{proposition}\label{prop:log-1}
    Assume that $O$ is diagonal. Then $r\in m^{\Q}$ if $1<\#J<m$, and $r\in n^{\Q}$ if $1<\#I<n$.
\end{proposition}
\begin{proof}
    Write $a=(a_1, a_2)$. Since $O$ is diagonal, there are $\delta_1,\delta_2\in\{-1,1\}$ such that 
    \[
        f(x,y) = (r\delta_1x+a_1, r\delta_2y+a_2), \quad (x,y)\in\R^2.
    \]
    From $\mu(K\cap f(K))>0$ we obtain
    \[
        \pi_1\mu(X\cap (r\delta_1 X+a_1))>0 \quad\text{and}\quad \pi_2\mu(Y\cap(r\delta_2 Y+a_2))>0.
    \]
    By Lemma~\ref{lem:deleted-embedding}, we can find similitudes $h_1, h_2$ such that $h_1(X)\subset X$ and $h_2(Y)\subset Y$. If $1<\#J<m$, then $Y=E(m,J)$ is a proper deleted-digit set with $\#J\geq 2$. By Remark~\ref{rem:ratio-log}, the similarity ratio of $h_2$ belongs to $r^{-1}m^{\Z}$. Combining this with  Lemma~\ref{lem:FW09}, we have $r\in m^{\Q}$. Similarly, if $1<\#I<n$, we can get $r\in n^{\Q}$ from $h_1(X)\subset X$.
\end{proof}

A direct corollary of Propositions~\ref{prop:log-1prime} and~\ref{prop:log-1} extends the similitude case of Algom and Hochman~\cite[Theorem~1.2]{AH19}.

\begin{corollary}
    Assume that $\frac{\log n}{\log m}\notin\Q$. Then $O$ is diagonal. Furthermore, if $1<\#J<m$ and $1<\#I<n$, then $r=1$.
\end{corollary}

For carpets that are not product sets, the diagonal case admits a stronger conclusion. We continue to write $M:=\diag(n,m)$ and use the subscript $\T$ to denote reduction modulo the appropriate integer lattice. The key ingredient is the following lemma, which shows that a diagonal affine self-embedding of $K$ must have equal logarithmic contraction rates with respect to the horizontal and vertical grid scales. The proof uses the idea from Proposition~\ref{prop:anti-product}.

\begin{lemma}
    Suppose $K\neq X\times Y$. Let $g$ be an invertible affine map on $\R^2$ with linear part $\diag(r_1,r_2)$. If $g(K)\subset K$, then $\frac{\log|r_1|}{\log n}=\frac{\log|r_2|}{\log m}$.
\end{lemma}
\begin{proof}
    Projecting to the coordinate axes, we have $0<|r_1|,|r_2|\leq 1$. Moreover, if $X$ (resp. $Y$) is a proper deleted-digit set, then $|r_1|\in n^{\Q}$ (resp. $|r_2|\in m^{\Q}$). Replacing $g$ by $g^2$ if necessary, we may assume that $r_1,r_2>0$. Let $\alpha:=\frac{\log (1/r_1)}{\log n}$ and $\beta:=\frac{\log (1/r_2)}{\log m}$. Equivalently, $r_1=n^{-\alpha}$ and $r_2=m^{-\beta}$. Suppose $\alpha\neq\beta$.
    
    Since $K\neq X\times Y$, we can choose $(x_0,y_0)\in (X\times Y)\cap(0,1)^2$ and $\delta>0$ such that 
    \[
        B_{\T}((x_0,y_0),\delta) \cap K_{\T} = \varnothing.
    \]
    Fix a large integer $L$ so that $m^{1-L}<\delta/2$. 

    \medskip
    \noindent\emph{Case 1: $\alpha<\beta$.} Recall that $\alpha\in\Q$ (resp. $\beta\in\Q$) whenever $X$ (resp. $Y$) is a proper deleted-digit set. Define $\Z_X=\Z^+$ (resp. $\Z_Y=\Z^+$) if $X=[0,1]$ (resp. $Y=[0,1]$) and $\Z_X$ (resp. $\Z_Y$) is the set of all positive integer $t$ such that $\alpha t\in \Z$ (resp. $\beta t\in \Z$) if $X$ (resp. $Y$) is proper. For each $t\in \Z_X\cap \Z_Y$,
    set $k_t=\lceil \beta t\rceil$ and consider $G_t:=M^{k_t}g^t$. Then 
    \begin{equation}\label{eq:newGt}
        G_t(x,y) = (n^{k_t-\alpha t}x+u_t, m^{k_t-\beta t}y+v_t), \quad (x,y)\in\R^2
    \end{equation}
    for some $u_t,v_t\in\R$. Note that $n^{k_t-\alpha t}\to\infty$ as $t\to\infty$ but $1\leq m^{k_t-\beta t}< m$. Moreover, since $g^t(K)\subset K$, we have $(G_t(K))_{\T}\subset (M^{k_t}K)_{\T}=K_{\T}$.
    
    We claim that $(\pi_1G_t(K))_{\T}=X_{\T}$ and $(\pi_2G_t(K))_{\T}=Y_{\T}$ for all large $t$. Indeed, if $X$ is a proper deleted-digit set, then $k_t-\alpha t$ is a nonnegative integer and hence 
    \[
       X_{\T}\supset (\pi_1G_t(K))_{\T}=X_{\T}+u_t.
    \]
    By Lemma~\ref{lem:compact-isometry}, $(\pi_1G_t(K))_{\T}=X_{\T}$. If $X=[0,1]$, then $n^{k_t-\alpha t}>1$ for all large $t$ and hence $(\pi_1G_t(K))_{\T}=X_{\T}$. The same argument also implies that 
    \begin{equation}\label{eq:cylinder-modulo}
        (\pi_1G_t\phi_w(K))_{\T} = X_{\T}, \quad \forall w\in\Gamma^L
    \end{equation}
    for all large $t$. Similarly, noticing that $k_t-\beta t=0$ if $Y$ is proper and $m^{k_t-\beta t}\geq 1$ if $Y=[0,1]$, one can show that $(\pi_2G_t(K))_{\T}=Y_{\T}$. 

    Fix any such $t$. Since $(\pi_2G_t(K))_{\T}=Y_{\T}$, there exists $z_t\in K$ with $\pi_2G_t(z_t)=y_0\pmod\Z$. Let $w$ be the length-$L$ prefix of a coding of $z_t$. By~\eqref{eq:cylinder-modulo}, there is $z_*\in\phi_w(K)$ such that $\pi_1G_t(z_*)=x_0\pmod\Z$. Since $z_t,z_*\in\phi_w(K)$, 
    \[
        \dist_{\T}(\pi_2G_t(z_*), y_0) \leq m^{k_t-\beta t}|\pi_2(z_*)-\pi_2(z_t)| \leq m^{1-L} < \delta/2.
    \]
    But this implies that $(G_t(z_*))_{\T}\in B_{\T}((x_0,y_0),\delta)$, contradicting $(G_t(K))_{\T}\subset K_{\T}$.

    \medskip
    \noindent\emph{Case 2: $\beta<\alpha$.} The case follows by interchanging the coordinate roles and taking $k_t=\lceil\alpha t\rceil$. The corresponding error is at most $n^{1-L}\leq m^{1-L}<\delta/2$.
\end{proof}

\begin{proposition}\label{prop:log-2}
    If $K\neq X\times Y$ and $O$ is diagonal, then $r=1$. In particular, $r=1$ if $J=[m]$ and $I=[n]$.
\end{proposition}
\begin{proof}
    Let $A:=\Pi^{-1}(K\cap f(K))$. By Lemma~\ref{lem:density}, there exists a cylinder density point $\omega\in A$. For $t\geq 1$, set $w_t:=\omega|_t$ and $A_t:=\{\eta\in\Sigma: w_t\eta\in A\}$. So 
    \[
        \nu(A_t) = \frac{\nu(A\cap[w_t])}{\nu([w_t])} \to 1, \quad t\to\infty.
    \]
    Since $r^{-1}O^{-1}$ commutes with $M$, there are $c_t\in\R^2$ such that  
    \[
        M^tf^{-1}\phi_{w_t}(z) = r^{-1}O^{-1}z+c_t, \quad \forall t\geq 1, z\in\R^2.
    \]
    By definition, for every $\eta\in A_t$ we have 
    \[
        r^{-1}O^{-1}\Pi(\eta)+c_t = M^tf^{-1}\phi_{w_t}(\Pi(\eta)) \in M^tK \subset K+\Z^2.
    \]
    Note that subtracting integer vectors from $c_t$ preserves this inclusion. 
    Passing to a subsequence, we may assume that $c_t\to c_* \pmod {\Z^2}$ for some $c_*\in[0,1)^2$. We claim that $r^{-1}O^{-1}K+c_*\subset K+\Z^2$. Otherwise, since $K+\Z^2$ is closed, some point of $K$ would have a neighborhood of positive $\mu$-measure whose images under the maps $z\mapsto r^{-1}O^{-1}z+c_t$ avoid $K+\Z^2$ for all large $t$. Its symbolic preimage would then avoid $A_t$, which contradicts $\nu(A_t)\to 1$.
 
    Only finitely many integer translates of $K$ intersect the set $r^{-1}O^{-1}K+c_*$. By Baire's theorem, there exists $e\in\Z^2$ such that $K+e$ contains a non-empty relatively open subset of $r^{-1}O^{-1}K+c_*$, say $r^{-1}O^{-1}\phi_w(K)+c_*\subset K+e$ for some $w\in\Gamma^*$. Let $k$ be the length of $w$. Writing $O=\diag(o_1,o_2)$, where $o_1,o_2\in\{\pm 1\}$, we have $g(K)\subset K$, where 
    \[
        g(x,y) = (r^{-1}o_1^{-1}n^{-k}x+u, r^{-1}o_2^{-1}m^{-k}y+v)
    \]
    for some $u,v\in\R$. It then follows from the previous lemma that $\frac{\log |r^{-1}n^{-k}|}{\log n}=\frac{\log |r^{-1}m^{-k}|}{\log m}$. So $\frac{\log r}{\log n}=\frac{\log r}{\log m}$. Since $n>m$, this holds only when $r=1$.

    If $I=[n]$ and $J=[m]$ then $X=Y=[0,1]$. Since $\Gamma\neq[n]\times[m]$, $K\neq X\times Y$. By Remark~\ref{rem:full-no-interchange}, $O$ is diagonal. The conclusion then follows from the first part.
\end{proof}

\begin{proof}[Proof of Theorem~\ref{thm:log}]
    By Theorem~\ref{thm:nonoblique}, $O$ is either diagonal or anti-diagonal. Proposition~\ref{prop:log-1prime} establishes the anti-diagonal case. For the diagonal case, since $K$ is not a segment, if neither $1<\#J<m$ nor $1<\#I<n$, then $J=[m]$ and $I=[n]$. The second assertion then follows from Propositions~\ref{prop:log-1} and~\ref{prop:log-2}.
\end{proof}

\bigskip
\noindent{\bf Acknowledgements.} 
Ruan is partially supported by NSFC grants Nos. 12371089 and 12531004, and the Fundamental Research Funds for the Central Universities of China grant No. 2024FZZX02-01-01. Xiao is partially supported by NSFC grant No. 12501113 and the Fundamental Research Funds for the Central Universities of China grant No. NS2025025. The project was initiated during Xiao's visit to Budapest in February 2026. He would like to thank Tam\'as Keleti for introducing him to the conjecture and for a stimulating discussion of the problem, during which several preliminary ideas related to the present work were explored.

\medskip
\noindent{\bf Declaration on AI assistance.}
Generative AI tools were used during the preparation of this manuscript to assist with language editing and proof checking, and to improve the clarity and organization of the exposition. All mathematical results, arguments, and proofs were developed and verified by the authors.

\small
\bibliographystyle{amsplain}

\end{document}